\documentclass[]{interact}

\usepackage{epstopdf}
\usepackage[caption=false]{subfig}

\usepackage[numbers,sort&compress]{natbib}
\bibpunct[, ]{[}{]}{,}{n}{,}{,}
\renewcommand\bibfont{\fontsize{10}{12}\selectfont}

\theoremstyle{plain}
\newtheorem{theorem}{Theorem}[section]
\newtheorem{lemma}[theorem]{Lemma}

\newtheorem{proposition}[theorem]{Proposition}

\theoremstyle{definition}

\theoremstyle{remark}
\newtheorem{remark}{Remark}

\usepackage{algorithm}
\usepackage{algorithmic}
\newcommand{\R}{{\mathbb R}}

\newcommand{\ee}{{\rm e}}

\newcommand{\e}{\varepsilon}
\newcommand{\la}{\langle}
\newcommand{\ra}{\rangle}

\DeclareMathOperator{\argmin}{argmin}

\def\vp{\varphi}

\newcommand{\norm}[1]{\left\lVert#1\right\rVert}
\def\one{{\boldsymbol{1}_n}}
\usepackage{multirow}
\usepackage{hyperref}

\graphicspath{ {./pics/} }

\usepackage[dvipsnames]{xcolor}
\def\po#1{{\color{black}#1}} 
\def\cu#1{{\color{black}#1}} 

\begin{document}


\title{Near-optimal tensor methods for minimizing the gradient norm of convex functions and accelerated primal-dual tensor methods}

\author{
    \name{
        Pavel Dvurechensky\textsuperscript{a}\thanks{CONTACT: Petr Ostroukhov: ostroukhov@phystech.edu, Pavel Dvurechensky: pavel.dvurechensky@wias-berlin.de},
        Petr Ostroukhov\textsuperscript{b,c,d},
         Alexander Gasnikov\textsuperscript{b,e,c},
        C\'esar A. Uribe\textsuperscript{f},
        Anastasiya Ivanova\textsuperscript{g,h}
    }
    \affil{
        \textsuperscript{a} Weierstrass Institute for Applied Analysis and Stochastics, Berlin, Germany;
        \textsuperscript{b} Moscow Institute of Physics and Technology, Dolgoprudny, Russia;
        \textsuperscript{c} Institute for Information Transmission Problems RAS, Moscow, Russia;
        \textsuperscript{d} Mohamed bin Zayed University of Artificial Intelligence, Abu-Dhabi, UAE;
        \textsuperscript{e} Skoltech, Moscow, Russia;
        \textsuperscript{f} Department of Electrical and Computer Engineering, Rice University, Houston TX, USA;
        \textsuperscript{g} HSE University, Moscow, Russian Federation;
        \textsuperscript{h} Univ. Grenoble Alpes, LJK, 38000 Grenoble, France
    }
}

\maketitle

\begin{abstract}
    \po{
    Motivated, in particular, by the entropy-regularized optimal transport problem, we consider convex optimization problems with linear equality constraints, where the dual objective has Lipschitz $p$-th order derivatives, and develop two approaches for solving such problems.
    The first approach is based on the minimization of the norm of the gradient in the dual problem and then the reconstruction of an approximate primal solution.   
    Recently, Grapiglia and Nesterov~\cite{grapiglia2020tensor} showed lower complexity bounds for the problem of minimizing the gradient norm of the function with Lipschitz~$p$-th order derivatives. Still, the question of optimal or near-optimal methods remained open as the algorithms presented in~\cite{grapiglia2020tensor} achieve suboptimal bounds only.
    We close this gap by proposing two near-optimal (up to logarithmic factors) methods with complexity bounds~$\tilde{O}(\e^{-2(p+1)/(3p+1)})$ and~$\tilde{O}(\e^{-2/(3p+1)})$ with respect to the initial objective residual and the distance between the starting point and solution respectively. 
    We then apply these results (having independent interest) to our primal-dual setting.
    As the second approach, we propose a direct accelerated primal-dual tensor method for convex problems with linear equality constraints, where the dual objective has Lipschitz~$p$-th order derivatives. 
    For this algorithm, we prove $\tilde O (\e^{-1 / (p + 1)})$ complexity in terms of the duality gap and the residual in the constraints.
    We illustrate the practical performance of the proposed algorithms in experiments on logistic regression, entropy-regularized optimal transport problem, and the minimal mutual information problem.
    }

\end{abstract}


\section{Introduction}\label{S:Intro}

    \cu{The idea of using higher-order derivatives in optimization methods has been known since the 1970's~\cite{hoffmann1978higher-order}, with increased interests recently~\cite{baes2009estimate,wibisono2016variational,birgin2017worst-case,cartis2017improved,agarwal2018lower,arjevani2018oracle}. Using high-order oracles has been shown to have better oracle complexities provably. However, their main practical bottleneck was the requirement of solving an auxiliary problem at each iteration that involved minimizing a regularized Taylor expansion of the objective, which in general, is a non-convex problem. Nesterov in~\cite{nesterov2021implementable} showed that an appropriate regularization makes the auxiliary problem convex. Moreover, he proposed an efficient method for solving the corresponding subproblem for the third-order method. This motivated a resurgence of research that introduced high-order (also referred to as tensor) methods for convex~\cite{gasnikov2019optimal,pmlr-v99-gasnikov19b,wilson2019accelerating,hinder2020near,bullins2019higher,bullins2018fast, nesterov2021inexact1, nesterov2021inexact2, kamzolov2020near} and non-convex settings~\cite{birgin2017worst-case,cartis2017improved}.}

    \po{
    In this paper, we consider a convex optimization problem with linear equality constraints of the form
    \begin{equation}
        \min_{\po{Ax = b}} f(x) \label{eq:PrSt},
    \end{equation}
    where~$f\po{: \R^n \to \R}$ is a convex function, \po{ $A \in \R^{m \times n}$, $b \in \R^m$}, and its dual
    \begin{equation}\label{eq:DualSt}
       \max_{\lambda \in \R^m} \left\{\vp(\lambda) := \la \lambda, b \ra + \max_{x\in Q} \left(- f(x) - \la A^\intercal \lambda,x \ra \right) \right\},
    \end{equation}
    where we assume that $\varphi$ has Lipschitz-continuous $p$-th derivative with constant~$M_p$.
    }
    

    \po{One way to tackle this problem is by minimizing the gradient norm of the dual function~\cite{nesterov2012make}.}
    In this case, the dual problem \eqref{eq:DualSt} is unconstrained, the gradient norm of the dual objective is equal to the primal constraints residual, and finding an~$\e$-stationary point for the dual problem allows one to find an approximate solution to the primal problem, see the details in Section~\ref{sec:motivation}. 

    Recently in~\cite{grapiglia2020tensor}, the authors considered a more general setting of composite convex optimization problems with H\"older-continuous higher-order derivatives and proposed a set of methods for finding approximate stationary points. As a particular case of~\cite[Corollary 5.8]{grapiglia2020tensor}, it follows that to find an~$\e$-stationary point~$\bar x$ such that $\|\nabla f(\bar x)\|_2 \leq \e$, their proposed method requires $$O \big( \big( {M_p R^p}/{\e} \big)^{{1}/{(p + 1)}}\big)$$  iterations with~$R$ being an estimate for the initial distance to the solution, i.e., ~$\|x_0-x^*\|_2 \leq R$. Moreover, as a particular case of~\cite[Corollary 5.10]{grapiglia2020tensor}, it follows that to find an \mbox{$\e$-stationary} point, their proposed method needs $$\tilde O  \big(  \big( {M_p \Delta_0^{p}}/{\e^{p + 1}} \big)^{{1}/{(p + 1)}}\big)$$ iterations with~$\Delta_0$ being an estimate for the initial objective residual,  i.e.,~$f(x_0)-f^* \leq \Delta_0$. However, \textit{these complexity bounds do not match the corresponding lower bounds obtained in~\cite[Theorem 6.6]{grapiglia2020tensor} and~\cite[Theorem 6.8]{grapiglia2020tensor} respectively,} where the number of iterations required to find an~$\e$-stationary point is of the order, respectively, 
   $$\Omega\big( \big( {M_p R^p}/{\e} \big)^{{2}/{(3p + 1)}}\big) \qquad \text{and} \qquad \Omega\big( \big( {M_p \Delta_0^p}/{\e^{p + 1}} \big)^{{2}/{(3p + 1)}}\big).$$
    
    \po{In~\cite{gasnikov2019optimal,bubeck2019near,pmlr-v99-gasnikov19b}, the authors considered finding~$\e$-approximate solution~$\bar x$ in terms of the objective residual, i.e., such that~$f(\bar x) - f^* \leq \e$. They proposed a class of near-optimal methods up to a logarithmic factor for \po{unconstrained minimization problem} in the general convex setting and under the additional assumption of uniform convexity.}
    \cu{We build upon the algorithm, developed in~\cite{bubeck2019near}, and propose methods for finding an approximate stationary point with near-optimal (up to a logarithmic multiplier) oracle complexities~\cite{grapiglia2020tensor}, see Table~\ref{tab:summary}.}

        \begin{table}[H]
        \tbl{Complexity of minimizing the gradient norm from~\cite{grapiglia2020tensor} and from this paper ($\Omega$ means ``lower bound'').}
        {
        \renewcommand{\arraystretch}{2.2}
        \begin{tabular}{cccc}
        \hline
        \textbf{Property} & \textbf{Lower bound ~\cite{grapiglia2020tensor}} & \textbf{Upper Bound~\cite{grapiglia2020tensor}} & \textbf{Upper Bound (this paper)} \\ 
        \hline \hline
          ~$f(x_0) -f^*\leq \Delta_0$ &~$\Omega \Big( \frac{M_p \Delta_0^p}{\e^{p + 1}} \Big)^\frac{2}{3p + 1}$ &~$\tilde O \Big( \frac{M_p \Delta_0^{p}}{\e^{p + 1}} \Big)^{\frac{1}{p + 1}}$ &~$\tilde O \Big( \frac{M_p \Delta_0^p}{\e^{p + 1}} \Big)^\frac{2}{3p + 1}$ \\
          \hline
         ~$\|x_0-x^*\|_2\leq R$ &~$\Omega \Big( \frac{M_p R^p}{\e} \Big)^\frac{2}{3p + 1}$ &~$O \Big( \frac{M_p R^p}{\e} \Big)^{\frac{1}{p + 1}}$  &~$\tilde O \Big( \frac{M_p R^p}{\e} \Big)^\frac{2}{3p + 1}$ \\ 
        \hline
        \end{tabular}}
    	\label{tab:summary}    
    \end{table}

    

    \po{In addition, our methods can be used for strongly convex high-order smooth functions, and we provide complexity estimations for finding $\e$-approximate stationary point for this case.
    Moreover, we explain how our methods can be extended to obtain near-optimal methods for functions with H\"older-continuous high-order derivatives.
    }
     
    An alternative approach \po{to tackle Problem \eqref{eq:PrSt}}, widely used in first-order methods~\cite{tran-dinh2014constrained,tran-dinh2015smooth,yurtsever2015universal,alacaoglu2017smooth,tran-dinh2020adaptive,chernov2016fast,dvurechensky2016primal-dual,anikin2017dual,dvurechensky2018computational,lin2019efficient,guminov2019accelerated,nesterov2020primal-dual,2019arXiv191000152L,2020arXiv200204783L}, constructs so-called primal-dual methods in which the main iterates are made for the dual problem with the goal being to find an~$\e$-approximate solution to the dual problem. Then, the information generated while the dual algorithm works is used, e.g., by averaging, to reconstruct an~$\e$-approximate solution to the primal problem. Motivated by such methods in the first-order setting, we propose an accelerated primal-dual high-order method that guarantees~$O(1/k^{p+1})$ decay after~$k$ iterations both for the primal-dual gap and linear constraints infeasibility. To our knowledge, this is the first high-order primal-dual accelerated method. In particular, we are not aware of any primal-dual second-order methods.
    
    This paper is organized as follows. We start in Section~\ref{sec:motivation} with examples to motivate finding approximate stationary points of convex functions. We describe the entropy-regularized optimal transport problem and show that its structure provides a natural justification for tensor methods that exploit the high-order smoothness properties of the corresponding dual problem. Section~\ref{sec:prelim} recalls some auxiliary results used later to prove our main results. Section~\ref{sec:residual} presents the near-optimal algorithm for finding approximate stationary points for the initial objective residual; near-optimal complexity bounds are shown explicitly. Section~\ref{sec:argument} shows the corresponding near-optimal algorithm for the initial variable residual; near-optimal complexity bounds are also shown. 
    In Section~\ref{sec:primal_dual}, we propose and prove convergence rate guarantees of our accelerated primal-dual tensor method for problems with linear equality constraints. Section~\ref{sec:numerics} shows some numerical results on the proposed algorithms for the logistic regression problem, entropy-regularized optimal transport problem, and minimization of ``bad'' functions, which give the lower bounds for the considered problem class. 
    \po{We provide numerical comparisons of the proposed tensor method for gradient norm minimization and primal-dual accelerated tensor method on entropy-regularized optimal transport and minimal mutual information problems.
    In addition, we compare numerically proposed algorithms with heuristical primal-dual modification of algorithm from \cite{bubeck2019near}.}
    Finally, conclusions and future work are presented in Section~\ref{sec:conclusions}.
    
   \textbf{Notation:} Let~$p \geq 1$. We denote by 
       ~$
        \nabla^p f(x)[h_1,...,h_p] 
       ~$
        the directional derivative of function~$f$ at~$x$ along directions~$h _i\in \R^n$,~$i = 1,...,p$.~$\nabla^p f(x)[h_1,...,h_p]~$ is a symmetric~$p$-linear form and its norm is defined as 
        \[
        \|\nabla^p f(x)\|_2 = \max_{h_1,...,h_p \in \R^n} \{ \nabla^p f(x)[h_1,...,h_p] : \|h_i\|_2 \leq 1, i = 1,...,p\},
        \]
        or equivalently
        \[
        \|\nabla^p f(x)\|_2 = \max_{h \in \R^n} \{ |\nabla^p f(x)[h,...,h]| : \|h\|_2 \leq 1\}.
        \]
       
        We denote $\|\cdot\|_2$ as the standard Euclidean norm, but our algorithm and derivations can be generalized for the Euclidean norm given by a general positive semi-definite matrix~$B$. In what follows, we also use notation $\nabla^p f(x) [h]^p \equiv \nabla^p f(x)[h,...,h]$.        We consider convex,~$p$ times differentiable on~$\R$ functions  satisfying Lipschitz condition for~$p$-th derivative
        \begin{equation}\label{eq:p-lipschitz}
        \|\nabla^p f(x) - \nabla^p f(y)\|_2 \leq M_p \|x-y\|_2, \qquad x,y \in \R^n.
        \end{equation}
        
        Given a function~$f$, numbers~$p \geq 1$ and~$M \geq 0$, define
        \begin{align}
            \Phi_{x, p}(y) &\triangleq \sum_{i = 0}^p \frac{1}{i!} \nabla^i f(x)[y - x]^i, \notag \\
            \Omega_{x, p, M}(y) &\triangleq \Phi_{x, p}(y) + \frac{M}{(p + 1)!} \|y - x\|_2^{p + 1}, \label{eq:regularized taylor approximation} \\
            T_{p,M}^f \left( x \right)&\in \mbox{Arg}\mathop {\min }\limits_{y\in \R^n} \Omega_{x, p, M}(y) \label{eq:tensor step}.
        \end{align}
       The main reason for using such regularization in \eqref{eq:regularized taylor approximation} is that, generally speaking, Taylor approximation of a convex function can be non-convex, and such regularization makes it convex~\cite{nesterov2021implementable} for sufficiently large $M$.
        Thus, from~\eqref{eq:p-lipschitz} and Taylor's theorem, it can be shown~\cite[Eqs. (2.6) and (2.7)]{birgin2017worst-case}, that
        \begin{align}
            |f(y) - \Phi_{x, p}(y)| & \le \frac{M_p}{p!} \|y - x\|^{p + 1}_2, \label{eq:objective and Taylor approximation diff upper bound} \\
            \|\nabla f(y) - \nabla \Phi_{x, p}(y)\|_2 & \le \frac{M_p}{(p - 1)!} \|y - x\|^{p}_2. \label{eq:grad objective and grad Taylor approximation diff upper bound}
        \end{align}
        
\section{A motivating example: problems with linear constraints}\label{sec:motivation}

    Let us consider a convex optimization problem with linear constraints
    \begin{equation}
    \label{eq:PrStGen}
        \min_{x \in Q \subseteq E} \left\{ f(x) : Ax = b \right\},
    \end{equation}
    where~$E$ is a finite-dimensional real vector space,~$Q$ is a simple closed convex set,~$A$ is a given linear operator from~$E$ to some finite-dimensional real vector space~$H$,~$b \in H$ is given,~$f$ is a convex function on~$Q$ with respect to some chosen norm~$\|\cdot\|_E$ on~$E$. 
    
    The Lagrange dual problem for \eqref{eq:PrStGen}, written as a minimization problem, is
    \vspace{-1mm}
    \begin{equation}
    \label{eq:DualPr}
    \min_{\lambda \in H^*} \left\{ \vp(\lambda) := \la \lambda, b \ra + \max_{x\in Q} \left(- f(x) - \la A^\intercal \lambda,x \ra \right) \right\}.
    \end{equation}
    
    We assume that the dual objective is smooth. In this case, by the Demyanov-Danskin theorem~\cite{danskin2012theory},~$\nabla \vp(\lambda) = b - Ax(\lambda)$,
    where 
    \[
    x(\lambda) := \arg \po{\max}_{x \in Q} \left(- f(x) - \la A^\intercal \lambda,x \ra \right).
    \] 

    \cu{Having the dual formulation in~\eqref{eq:DualPr} at hand, the following proposition justifies why minimizing the norm of a gradient is useful in the convex setting. }
    
    \begin{proposition}[Lemma 1 in~\cite{gasnikov2016efficient}]\label{prop:dual gradient norm and primal solution}
    Let $\lambda \in H^*$ be such that $-\la \lambda, \nabla \vp(\lambda) \ra \leq \e_f$, and $\|\nabla \vp(\lambda)\|_H \leq \e_{eq}$. Then
    \begin{equation}\label{eq:e_f_e_eq-solution}
        f(x(\lambda)) - f^* \leq \e_f, \quad \|Ax(\lambda) - b\|_H\leq \e_{eq}. 
    \end{equation}
    \end{proposition}
    
    Proposition~\ref{prop:dual gradient norm and primal solution} implies that if there is a method for the dual Problem \eqref{eq:DualPr} and this method generates a bounded sequence of iterates~$\lambda_k$ and a point~$\lambda_k$ s.t., the gradient of the dual objective is small, then, using the relation~$x(\lambda_k)$ we can reconstruct an approximate solution to the primal problem. This is the general motivation for convex optimization methods for minimizing the objective gradient norm. Moreover, the complexity bound for the dual method directly translates to the complexity for solving the primal problem without any overhead.
    
    To further motivate the high-order methods to minimize the objective gradient norm, we present a particular example of a smooth dual objective with Lipschitz continuous high-order derivatives.
    This example is the Entropy-regularized optimal transport problem~\cite{cuturi2013sinkhorn,cuturi2016smoothed}. Next, we briefly describe the problem and the properties of the dual objective.
    
    Consider two histograms~$p,q \in \Sigma_n$ on a support of size~$n$, where~$\Sigma_n$ is the standard simplex. Also, consider a matrix~$M \in \mathbb{R}^{n \times n}_{+}$ which is symmetric and accounts for the ``cost'' of transportation such that~$M_{ij}$ is the cost of moving a unit of mass from bin~$i$ to bin~$j$ in the corresponding supports of distributions~$p$ and~$q$. For example, given support points~$(x_i)_{1\leq i \leq n}$ on the Euclidean space, one can consider~$M_{ij} = \|x_i- x_j\|_2^2$, which corresponds to 2-Wasserstein distance. 
    
    The entropy-regularized optimal transport problem is defined as:
    \begin{align}\label{eq:wass}
        W_{\gamma}(p,q) & \triangleq \min_{X \in U(p,q)} \{  \langle M, X \rangle - \gamma E(X)\},
    \end{align}
    where~$\langle M, X \rangle$ is the Frobenius dot-product,~$\gamma \geq 0$ is a regularization parameter,~$E(X) \triangleq - \sum_{i,j}X_{ij}\ln(X_{ij})$, and~$U$ is the transportation polytope defined as
    \begin{align*}
        U(p,q) \triangleq \{ X \in \mathbb{R}_{+}^{n\times n} \mid X \boldsymbol{1}_n = p, X^\intercal\boldsymbol{1}_n=q  \}.
    \end{align*}
    It is known that Problem~\eqref{eq:wass}  is strongly convex and admits a unique optimal solution~$X^*$~\cite{cuturi2016smoothed}. If~$\gamma=0$ and~$M_{ij} = \|x_i-x_j\|_2^r$, $W_{\gamma}(p,q)$ in \eqref{eq:wass} is known as the~$r$-th power of the~$r$-Wasserstein distance between~$p$ and~$q$. 
    
    A standard way to deal with~\eqref{eq:wass} is to write its dual as follows:  
    \begin{align}
        \min_{X \in U(p,q)} &\la M, X \ra + \gamma \la X, \ln X \ra  \notag \\
        &=\min_{X \in \Sigma_{n^2}} \la M, X \ra + \gamma \la X, \ln X \ra + \max_{\xi,\eta} \left\{\la \xi,p-X\one\ra + \la \eta, q -X^\intercal\one \ra \right\}  \notag \\
    &=\max_{\xi,\eta} \left\{  \la \xi,p \ra + \la \eta, q \ra  +  \min_{X \in \Sigma_{n^2}}\left\{\la M + \xi \one^\intercal + \one \mu^\intercal + \gamma  \ln X,  X \ra \right\}   \right\} \notag \\
    &=\max_{{\xi},{\eta}} - \gamma {\ln} \sum_{i,j=1}^n \exp \left(-\frac{1}{\gamma}(M_{ij}-{\xi_i}-{\eta_j}) \right) + \la{\xi},{p}\ra + \la {\eta}, {q}\ra. \label{eq:OT-PD} 
    \end{align}
    In this case, the explicit dependence of the primal solution from the dual variables is given by
    \begin{equation}
        X(\xi,\eta) = \frac{{\rm diag}\big(\ee^{\frac{\xi}{\gamma}}\big)\ee^{-\frac{M}{\gamma}}{\rm diag}\big(\ee^{\frac{\eta}{\gamma}}\big)}{\ee^{\frac{\xi}{\gamma}}\ee^{-\frac{M}{\gamma}}\ee^{\frac{\eta}{\gamma}}},
    \end{equation}
    where the function~$\ee(\cdot)$ indicates component-wise exponentiation of vectors and matrices, i.e., ~$[\ee(A)]_{ij} = \exp(A_{ij})$. Also, for a vector~$a$,~${\rm diag}(a)$ denotes a diagonal matrix with the vector~$a$ on the diagonal.
    We underline that as opposed to the standard dual problem derived in \cite{cuturi2013sinkhorn}, we consider~$X$ to lie not in~$\R^{n \times n}_+$, but rather in the standard simplex of the size~$n^2$, the latter being the corollary of the marginal constraints~$X \boldsymbol{1}_n = p$,~$X^\intercal\boldsymbol{1}_n=q$ since~$p,q \in \Sigma_n$. This allows us to obtain a high-order smooth dual objective with a softmax form, as we will show next. On the contrary, the dual problem in~\cite{cuturi2013sinkhorn} has a sum of exponents in the dual objective, meaning that the derivatives are not Lipschitz-continuous.
    
    To show the correspondence to a general primal and dual pair of Problems~\eqref{eq:PrStGen}--\eqref{eq:DualPr}, let us assume without loss of generality that~$E = \R^{n^2}$,~$\|\cdot\|_E = \|\cdot\|_1$, and variable~$x = {\rm vec}(X) \in \R^{n^2}$ to be the vector obtained from a matrix~$X$ by writing each column of~$X$ below the previous column. 
    For the dual space we consider~$H = \R^{2n}$,~$\| \cdot \|_H = \| \cdot \|_2$.
    Also we set~$f(x) = \la M,X\ra + \gamma \la X, \ln X\ra$,~$Q=\Sigma_{n^2}$,~$b^\intercal = (p^\intercal,q^\intercal)$,~$A:\R^{n^2}\to \R^{2n}$ defined by the identity~$(A\,{\rm vec}(X))^\intercal = ((X \boldsymbol{1}_n)^\intercal,(X^\intercal \boldsymbol{1}_n)^\intercal)$, and~$\lambda^\intercal = (\xi^\intercal,\eta^\intercal)$. 
    Note that the matrix~$A$ has the form
    \[
    A = \left(
    \begin{aligned}
    I_n \quad & I_n & I_n  \quad& ...\\
    \boldsymbol{1}_n^\intercal\quad & \boldsymbol{0}_n^\intercal \quad& \boldsymbol{0}_n^\intercal\quad & ...\\
    \boldsymbol{0}_n^\intercal\quad & \boldsymbol{1}_n^\intercal \quad& \boldsymbol{0}_n^\intercal\quad & ...\\
    ...\quad & ...\quad & ... \quad & ...\\
    \end{aligned}
    \right),
    \]
    where~$I_n$ is the identity matrix,~$\boldsymbol{0}_n^\intercal$ is the vector of all zeros.
    Using these notations, we can write the dual problem in \eqref{eq:OT-PD} as
    \begin{align}
       & \max_{\lambda} - \gamma {\ln} \sum_{i,j=1}^n \exp \left(-\frac{[M-A^\intercal \lambda]_{ij}}{\gamma} \right) + \la \lambda, b\ra  \notag \\
       &= \max_{\lambda} - \textbf{smax}_\gamma(A^\intercal\lambda-M) + \la \lambda, b\ra \\
       &= \min_{\lambda} \textbf{smax}_\gamma(A^\intercal\lambda-M) - \la \lambda, b\ra
       \label{eq:OT_dual_smax}
    \end{align}
    where
    \begin{align}\label{eq:softmax}
        \textbf{smax}_\gamma(y) \triangleq \gamma \log \left( \sum_{i=1}^m \exp(y_i/ \gamma) \right).
    \end{align}
    More importantly, the following property holds.
    \begin{proposition}[ {\cite[Theorem 3.4]{bullins2019higher}    }]\label{them:smax}
        Let~$z \in \R^n$,~$c \in \R^m$ and~$\mathcal{A}: \R^n \to \R^m$. Then the function~$\textbf{smax}_\gamma(\mathcal{A}z-c)$ is (order~$3$)~$\frac{15}{\gamma^3}$-smooth with respect to~$\|\cdot\|_{\mathcal{A}^\intercal\mathcal{A}}$.
    \end{proposition}
    As a corollary, the dual objective in \eqref{eq:OT_dual_smax} has~$({15}/{\gamma^3})$-Lipschitz-continuous third derivative w.r.t.~$\|\cdot\|_{AA^\intercal}$. Equivalently to the first order Lipschitz constant from~\cite{nesterov2005smooth}, we can write third order Lipschitz constant from Proposition~\ref{them:smax} as~$15 \|A\|_{E \to H}^4 / \gamma^3$. Due to our choice of norms in~$E$ and~$H$,~$\|A\|_{E \to H}$ is equal to the maximal Euclidean norm of the columns of matrix A. Thus,~$\|A\|_{E \to H} = \sqrt 2$.
    
    We can conclude that minimizing the norm of the gradient of the dual objective allows one to obtain an approximate solution to the corresponding primal problem that estimates the optimal transport cost and optimal transportation plan in this case. Thus, having a fast method that exploits the high-order smoothness of the dual problem can provide efficient algorithms for the computation of Sinkhorn distance~\cite{cuturi2013sinkhorn} defined as the solution to entropy regularized optimal transport problem. 
    
\section{Preliminaries}
\label{sec:prelim}

\cu{In this section, we present a series of auxiliary results that will later enable the development of our near-optimal algorithms, which will be presented in Section~\ref{sec:algos}. The reader might skip this section and revisit it for proof details.}

    We measure the complexity of algorithms in the number of calls to the oracle of the objective function. By oracle of some objective~$f$ we mean some mapping~$x \mapsto \{f(x), \nabla f(x), ..., \nabla^p f(x) \}, \forall x \in \text{dom} f,\ p \ge 1$.
    
    To make the paper self-contained, in this section, we recall the near-optimal tensor methods for minimization of convex objective functions with Lipschitz-continuous~$p$-th derivative \po{\cite{bubeck2019near}}.

\po{
    \floatname{algorithm}{Algorithm}
        \begin{algorithm}
        	\caption{Accelerated Taylor Descent~\cite[Algorithm 1]{bubeck2019near}}
        	\label{alg:MSN}
        	\begin{algorithmic}[1]
        		\REQUIRE~$N$~--- iteration number.
                \STATE Set~$A_0 = 0$, $x_0 = y_0 = 0$
        		\FOR{$k=0,1,2,\ldots, N-1$}
        		\STATE Compute $\lambda_{k + 1} > 0$ and $y_{k + 1} \in \R^d$ such that
        		\begin{equation}
        		\label{eq:L_k-cond}
                \frac{1}{2} \le \lambda_{k + 1} \frac{M_p \|y_{k + 1} - \tilde x_k\|^{p - 1}}{(p - 1)!} \le \frac{p}{p + 1},
        		\end{equation}
        		where
                \begin{gather*}
                    a_{k + 1} = \frac{\lambda_{k + 1} + \sqrt{\lambda_{k + 1}^2 + 4 \lambda_{k + 1} A_k}}{2},\quad A_{k + 1} = A_k + a_{k + 1}, \\
                    \tilde x_k = \frac{A_k}{A_{k + 1}} y_k + \frac{a_{k + 1}}{A_{k + 1}} x_k,\quad y_{k + 1} = T_{p, pM_p}^f(\tilde x_k).
                \end{gather*}
        		\STATE~$x_{k + 1} = x_k - a_{k + 1} \nabla f(y_{k + 1}).$
        		\ENDFOR
        		\RETURN~$y_N$
        	\end{algorithmic}
        \end{algorithm}
}
    

    \po{
    \begin{theorem}[Theorem 1 in \cite{bubeck2019near}]\label{Th:main-conv}
        Let $f$ be a convex function with $M_p$-Lipschitz $p$-th derivative. 
        Assume, that exists $R > 0: \| x_0 - x^*\|_2 \le R$, and let \mbox{$c_p = {2^{p - 1}(p + 1)^\frac{3p + 1}{2}}/{(p - 1)!}$}.
        Then, for all $N\geq 0$, the output of Algorithm \ref{alg:MSN} has the following property
                \begin{equation}
            f(y_N) - f(x^*) \le \frac{c_p M_p R^{p + 1}}{N^\frac{3p + 1}{2}}.
        \end{equation}
    	Moreover, each iteration~$k$ requires~$O\left(\ln({1}/{\e})\right)$ oracle calls. 
    \end{theorem}
    }
    
    At the core of the result in Theorem~\ref{Th:main-conv}, the authors in~\cite{bubeck2019near} use the following auxiliary result that we will later use in our proofs. We restate this result for completeness. 
    \po{
    \begin{lemma}[Lemma 11 from~\cite{bubeck2019near}]
        Let $c_p ={2^{p - 1} (p + 1)^\frac{3p + 1}{2}}/{(p - 1)!}$, and $k\geq 0$. Then, $A_k$ from Algorithm~\ref{alg:MSN} has the following property
        \begin{equation}\label{eq:A_k upper bound}
            A_k \ge \frac{1}{c_p M_p R^{p - 1}} k^\frac{3p + 1}{2}.
        \end{equation}
    \end{lemma}
    }
    The following statement from~\cite{monteiro2013accelerated} also holds for Algorithm~\ref{alg:MSN}.
    
    \begin{theorem}[Theorem 3.6 in~\cite{monteiro2013accelerated}]
        Let sequence \po{$(x_k, \tilde x_k, y_k), k \ge 0$} be generated by Algorithm~\ref{alg:MSN}, and define~$R:= \|\po{x_0} - x^*\|_2$. Then for all~$N \ge 0$
        \begin{gather}
            \frac{1}{2} \|\po{x_N} - x^*\|_2^2 + A_N \left( f(y_N) - f(x^*) \right) + \frac{1}{4} \sum_{k = 1}^N A_k L_{k - 1} \|y_k - \po{\tilde x_{k - 1}}\|_2^2 \le \frac{R^2}{2}, \\
            f(y_N) - f(x^*) \le \frac{R^2}{2 A_N}, \quad \|\po{x_N} - x^*\|_2 \le R, \label{eq:residual upper bound by A_N} \\
            \sum_{k = 1}^N A_k L_{k - 1} \|y_k - \po{\tilde x_{k - 1}}\|_2^2 \le 2 R^2.
        \end{gather}
    \end{theorem}
    
    Algorithm~\ref{alg:MSN} requires intermediate steps to find~$L_k$ and~$y_{k + 1}$. Since they depend on each other, we need to find them iteratively with a binary search procedure described  in~\cite[Section 4]{bubeck2019near}.

    \po{
    Since we know that $\Omega_{x,p,M}(y)$ is convex, then $\tilde z = T_{p, M}^f(x)$ exists. Thus,
    }
    \begin{equation}\label{eq:additional condition for solution}
        \po{\forall x \in \R^n \Rightarrow} \Omega_{x, p, M}(\tilde z) \overset{\po{\eqref{eq:tensor step}}}{=} \min_{y\in \R^n} \Omega_{x, p, M}(y)\le \Omega_{x, p, M}(x) \overset{\po{\eqref{eq:regularized taylor approximation}}}{=} f(x).
    \end{equation}
    We use this inequality to prove the following lemma, which is a particular case of~\cite[Lemma 5.2]{grapiglia2020tensor} with~$\nu=1$,~$\theta = 0$, and~$\vp = 0$.
    \begin{lemma}[Lemma 5.2 in~\cite{grapiglia2020tensor}]\label{lem:grap_nest}
        Let $p \ge 1$,~$M_p < \infty$,~$M \geq (p + 2)M_p$ and let for some~$x \in \mathbb R^n$
        \[
            \tilde z = T_{p, M}^{f}(x).
        \]
        Then,
        \[
            f(x) - f(\tilde z) \geq \frac{1}{4 (p + 2)! M^\frac{1}{p}} \|\nabla f(\tilde z)\|_2^\frac{p + 1}{p}.
        \]
    \end{lemma}
    \begin{proof}
        From triangle inequality, \eqref{eq:grad objective and grad Taylor approximation diff upper bound} and definition of~$\tilde z$, we get
        \begin{align}
            \|\nabla f(\tilde z)\|_2 &= \|\nabla f(\tilde z) - \nabla \Phi_{x, p}(\tilde z) + \nabla \Phi_{x, p}(\tilde z) - \nabla \Omega_{x, p, M}(\tilde z) + \nabla \Omega_{x, p, M}(\tilde z) \|_2 \notag \\
            &\le \|\nabla f(\tilde z) - \nabla \Phi_{x, p}(\tilde z)\|_2 + \|\nabla \Phi_{x, p}(\tilde z) - \nabla \Omega_{x, p, M}(\tilde z)\| + \|\nabla \Omega_{x, p, M}(\tilde z)\|_2\notag \\
            &\overset{\eqref{eq:grad objective and grad Taylor approximation diff upper bound}}{\le} \frac{M_p}{(p - 1)!}\|\tilde z - x\|_2^p + \frac{M}{p!}\|\tilde z - x\|_2^p = \left( \frac{p M_p}{p!} + \frac{M}{p!} \right) \|\tilde z - x\|_2^p \notag \\
            &\le 2 M \|\tilde z - x\|_2^p. \label{eq:lem_1 eq_1}
        \end{align}
        Next, from \eqref{eq:objective and Taylor approximation diff upper bound}, \eqref{eq:additional condition for solution} follows
        \begin{align*}
            f(\tilde z) &\overset{\eqref{eq:objective and Taylor approximation diff upper bound}}{\le} \Phi_{x, p} (\tilde z) + \frac{M_p}{p!} \|\tilde z - x\|^{p + 1}_2 = \Phi_{x, p} (\tilde z) + \frac{(p + 1) M_p}{(p + 1)!} \|\tilde z - x\|^{p + 1}_2 \\
            &= \Phi_{x, p} (\tilde z) + \frac{M}{(p + 1)!} \|\tilde z - x\|_2^{p + 1} - \frac{(M - (p + 1)M_p)}{(p + 1)!} \|\tilde z - x\|_2^{p + 1} \\
            &= \Omega_{x, p, M}(\tilde z) - \frac{(M - (p + 1)M_p)}{(p + 1)!} \|\tilde z - x\|_2^{p + 1} \\
            &\overset{\eqref{eq:additional condition for solution}}{\le} f(x) - \frac{(M - (p + 1)M_p)}{(p + 1)!} \|\tilde z - x\|_2^{p + 1}.
        \end{align*}
        Since~$M \ge (p + 2)M_p \Leftrightarrow \frac{1}{p + 2}M \ge M_p$, we get
        \begin{align}
            f(x) - f(\tilde z) &\ge  \frac{(M - (p + 1)M_p)}{(p + 1)!} \|\tilde z - x\|_2^{p + 1} \notag \\
            &\ge \frac{(M - \frac{p + 1}{p + 2}M)}{(p + 1)!} \|\tilde z - x\|_2^{p + 1} \notag \\
            &= \frac{M}{(p + 2)!} \|\tilde z - x \|_2^{p + 1}. \label{eq:lem_1 eq_2}
        \end{align}
        If we combine \eqref{eq:lem_1 eq_1} and \eqref{eq:lem_1 eq_2}, we obtain the final result for all~$p \ge 1$
        \begin{align*}
            f(x) - f(\tilde z) &\ge \frac{M}{(p + 2)!} \|\tilde z - x \|_2^{p + 1} = \frac{M}{(p + 2)!} (\|\tilde z - x \|_2^{p})^{(p + 1) / p} \\
            &\overset{\eqref{eq:lem_1 eq_1}}{\ge} \frac{M}{(p + 2)!} \left( \frac{\|\nabla f(\tilde z)\|_2}{2 M} \right)^\frac{p + 1}{p} = \frac{\|\nabla f(\tilde z)\|_2^\frac{p + 1}{p}}{2^{\frac{p + 1}{p}} M^{\frac{1}{p}} (p + 2)!} \\
            &\ge \frac{\|\nabla f(\tilde z)\|_2^\frac{p + 1}{p}}{4 M^{\frac{1}{p}} (p + 2)!}.
        \end{align*}
    \end{proof}

\section{Near-optimal tensor methods for gradient norm minimization}\label{sec:algos}

\cu{In this section, we will build upon Algorithm~\ref{alg:MSN} to develop near-optimal tensor methods for gradient norm minimization of convex functions. This section is divided into two parts: first, we develop near-optimal tensor methods with respect to an estimate of the initial objective residual in Subsection~\ref{sec:residual} presented in Algorithm~\ref{alg:Delta}, then in Subsection~\ref{sec:argument}, we develop near-optimal tensor methods with respect to an estimate of the initial argument residual presented in Algorithm~\ref{alg:R}. Note that both proposed algorithms have Algorithm~\ref{alg:MSN} at their core, and rely on the bounds presented in Section~\ref{sec:prelim}. }

    \subsection{Near-optimal tensor methods with respect to the initial objective residual}\label{sec:residual}

        In this subsection, we build up from Algorithm~\ref{alg:MSN}        to develop a near-optimal algorithm for which we can provide explicit complexity bounds for approximating a stationary point. The obtained oracle complexity bound matches up to a logarithmic factor the lower complexity bound presented in~\cite{grapiglia2020tensor}. This subsection focuses on complexity bounds that depend on the initial objective residual. Thus, the basic assumption is that the starting point~$x_0$ satisfies~$f(x_0) - f^* \le \Delta_0$. 
        
        \begin{algorithm}
        	\caption{Near-optimal algorithm with respect to initial objective residual}\label{alg:Delta}
        	\begin{algorithmic}[1]
        		\REQUIRE $p \ge 2$,~$M_p$, $x_0$,~$\Delta_0: f(x_0) - f^* \le \Delta_0$,~$\e >0$.
        		
        		\STATE \textbf{Define}:
                \begin{gather*}
                    k = 0, \quad M_\mu = (p + 2)M_p, \quad \mu = \frac{\e^2}{32 \Delta_0}, \quad \tilde \e = \frac{(\e/2)^{\frac{p+1}{p}}}{4 (p+2)! M_\mu^{\frac{1}{p}}},\\
                    z_0 = x_0,   \quad f_\mu(x) = f(x) + \frac{\mu}{2}\|x - x_0\|_2^2.
                \end{gather*}
                
        		
        		\WHILE {$\Delta_k \ge \tilde \e$\, where~$\Delta_k = \Delta_0\cdot2^{-k}$.}
                
        		\STATE Set~$z_{k + 1}$ as the output of Algorithm~\ref{alg:MSN}  applied to~$f_{\mu}(x)$ starting from~$z_k$ and run for~$N_k$ steps, where~$N_k$ is such that~$A_{N_k} \ge {2}/{\mu}$.
        		\STATE $k = k + 1$.
        		\ENDWHILE
        		\STATE \textbf{Find}~$\tilde z = T_{p,\ M_\mu}^{f_\mu}(z_k)$.
        		\RETURN~$\tilde z$.
        	\end{algorithmic}
        \end{algorithm}
        
        \begin{theorem}\label{thm:func residual}
        	Let~$p \ge 2$. Assume the function~$f$ is convex,~$p$ times differentiable on~$\R^n$ with~$M_p$-Lipschitz~$p$-th derivative.
            \po{Assume, that $\Delta_0 > 0$ is such that \mbox{$f(x_0) - f^* \le \Delta_0$}.}
            Let~$\tilde z$ be generated by Algorithm~\ref{alg:Delta}. Then
        	\[
        	\|\nabla f(\tilde z)\|_2 \le \e,
        	\]
        	and the total number of iterations of Algorithm~\ref{alg:MSN} required by Algorithm~\ref{alg:Delta} is 
        	\[
        	O\Bigg(\frac{M_p^\frac{2}{3p + 1}}{\e^\frac{2(p + 1)}{3p + 1}} \Delta_0^\frac{2p}{3p + 1} + \log_2 \frac{2^\frac{4p - 3}{p + 1} \Delta_0 (pM_p)^\frac{1}{p}(p + 1)!}{\e^\frac{p}{p+1}} \Bigg).
        	\]
        	Moreover, the total oracle complexity is within a~$O\left(\ln\frac{1}{\e}\right)$ factor of the above iteration complexity due to the use of binary search in Algorithm~\ref{alg:MSN}.
        \end{theorem}
        \begin{proof}
           
            By definition of~$f_\mu(x)$:
        	\[
        	f_\mu(x_0) - f_\mu(x_\mu^*) = f(x_0) - f(x_\mu^*) - \frac{\mu}{2}\|x_\mu^* - x_0\|_2^2  \leq f(x_0) - f(x^*) \leq \Delta_0,
        	\]
        	Where~$x_\mu^*$ is the minimum of~$f_{\mu}(x)$. So, for~$k = 0$ we have~$f_\mu(z_k) - f_\mu(x_\mu^*) \le \Delta_k$.
        	Let us assume that~$f_\mu(z_k) - f_\mu(x_\mu^*) \le \Delta_k$ and show that~$f_\mu(z_{k+1}) - f_\mu(x_\mu^*) \le \Delta_{k+1}$. 
            As you can see, we use Algorithm~\ref{alg:MSN} inside Algorithm~\ref{alg:Delta} and restart it every time~$A_{N_k} \ge {2}/{\mu}$. We can do this due to the strong convexity of~$f_\mu$.
            From \eqref{eq:residual upper bound by A_N}, strong convexity and this restart condition it holds that
            \begin{align}
                f_\mu(z_{k + 1}) - f_\mu(x_\mu^*) 
                &\overset{\eqref{eq:residual upper bound by A_N}}{\le} \frac{\|z_k - x^*_\mu\|^2_2}{2 A_{N_k}} \le \frac{1}{2 A_{N_k}} \left( \frac{2(f_\mu(z_k) - f_\mu(x_\mu^*))}{\mu} \right) \notag \\
                &\le \frac{\Delta_k}{\mu A_{N_k}} \le \frac{\Delta_{k}}{2} = \Delta_{k + 1}. \label{eq:residual Delta upper bound from the stopping criterion}
            \end{align}
        	Thus,~$f_\mu(z_k) - f_\mu(x_\mu^*) \le \Delta_k$ for all~$k \ge 0$.

            Although such a stopping criterion is useful for numerical experiments, it is not obvious how to derive a theoretical upper bound for the number of steps of Algorithm~\ref{alg:MSN}. 
            We can use \eqref{eq:A_k upper bound} and \eqref{eq:residual Delta upper bound from the stopping criterion} to obtain an upper bound for the number $\tilde N_k$ of steps of Algorithm~\ref{alg:MSN} sufficient to fulfill this stopping criterion. Denote by~$A_{\tilde N_k}$ such constant~$A_N$, which we get after~$\tilde N_k$ steps of the Algorithm~\ref{alg:MSN}. 
            Then
            \[
                f_\mu(z_{k + 1}) - f_\mu(x_\mu^*) \overset{\eqref{eq:residual Delta upper bound from the stopping criterion}}{\le} \frac{\Delta_k}{\mu A_{\tilde N_k}}.
            \]
            From strong convexity, we know that
            \[
                \|z_k - x^*_\mu\|_2 \le \sqrt{\frac{2}{\mu}\Delta_k}.
            \]
            Thus, we can choose~$R_k = \sqrt{({2}/{\mu})\Delta_k}$.
            
            From \eqref{eq:A_k upper bound} we can choose~$N_k$ to fulfill the stopping criterion~$A_{\tilde N_k} \ge {2}/{\mu}$:
            \begin{equation}\label{eq:tilde N_k for Delta}
                \tilde N_k = \max\left\{\left\lceil \left( \frac{2c_{\po{p}}M_p 2^{\frac{p+1}{2}}}{\mu^{\frac{p+1}{2}}} \Delta_k^{\frac{p-1}{2}} \right)^{\frac{2}{3p+1}}\right\rceil,1 \right\}.
            \end{equation}
            Therefore, we get
            \begin{equation*}
                f_\mu(z_{k + 1}) - f_\mu(x_\mu^*) \overset{\eqref{eq:residual Delta upper bound from the stopping criterion}}{\le} \frac{\Delta_k}{\mu A_{\tilde N_k}} \le \frac{c_{\po{p}} M_p}{\tilde N_k^\frac{3p + 1}{2}} \left( \frac{2 \Delta_k}{\mu} \right)^\frac{p + 1}{2} \le \frac{\Delta_k}{2}.
            \end{equation*}
        
            Next, we estimate~$\|\nabla f_\mu(\tilde z)\|_2$. \po{Since $\tilde z = T_{p, M_\mu}^{f_\mu}(z_k)$, and} according to {Lemma~\ref{lem:grap_nest}}, we have
        	\begin{equation}\label{aux1}
                f_\mu(z_k) - f_\mu(\tilde z) \ge \frac{1}{4 (p + 2)! M_\mu^\frac{1}{p}} \|\nabla f_{\po{\mu}}(\tilde z)\|_2^\frac{p + 1}{p}.
        	\end{equation}
        	At the same time, by the stopping criterion in Algorithm~\ref{alg:Delta},
        	\begin{equation}\label{aux2}
        	f_\mu(z_k) - f_\mu(\tilde z) \le f_\mu(z_k) - f_\mu(x_\mu^*) \le \Delta_k \leq \tilde \e.
        	\end{equation}
        	By the definition of~$\tilde \e$ and \eqref{aux1}, \eqref{aux2}, we have that
        	\begin{equation}\label{eq:grad_est}
        	\|\nabla f_\mu(\tilde z)\|_2 \le \frac{\e}{2}.
        	\end{equation}
        	\po{
            Since the right-hand side of \eqref{aux1} is non-negative, we can state that 
            \begin{equation}\label{eq:aux2}
                f_\mu(\tilde z) \le f_\mu(z_k).
            \end{equation}
            }
            By definition,~$f_\mu$ is~$\mu$-strongly convex and, using \eqref{eq:aux2}, we get
        	\begin{align}
        	   &\frac{\mu}{2} \|x_\mu^* - x_0\|^2_2 \le f_\mu(x_0) - f_\mu(x_\mu^*) \le \Delta_0, \label{eq:mu_x_0} \\
        	   &\frac{\mu}{2} \|\tilde z - x_\mu^*\|^2_2 \le f_\mu(\tilde z) - f_\mu(x_\mu^*) \overset{\eqref{eq:aux2}}{\le} f_\mu(z_k) - f_\mu(x_\mu^*) \leq \po{\Delta_k \le} \Delta_0. \label{eq:mu_tilde_z}
        	\end{align}
        	Applying triangle inequality to the sum of \eqref{eq:mu_x_0} and \eqref{eq:mu_tilde_z}, we get
        	\begin{gather*}
        	\frac{\mu}{2} \|\tilde z - x_0\|^2_2 \le \mu \big(\|x_\mu^* - x_0\|^2_2 + \|\tilde z - x_\mu^*\|^2_2 \big) \le 4 \Delta_0,
        	\end{gather*}
        	and
        	\begin{gather*}
        	   \|\tilde z - x_0\|_2 \le 2 \sqrt{\frac{ 2\Delta_0}{\mu}}.
        	\end{gather*}
        	By definition of~$\mu$ in Algorithm~\ref{alg:Delta}, we have
        	\begin{equation}\label{eq:arg_residual}
        	\mu \|\tilde z - x_0\|_2 \le \mu \cdot 2 \sqrt{\frac{ 2\Delta_0}{\mu}} = 2 \sqrt{2\mu \Delta_0} = \frac{\e}{2}.
        	\end{equation}
        	
        	Finally, according to the definition of~$f_\mu$, \eqref{eq:grad_est}, \eqref{eq:arg_residual} and triangle inequality, we get
        	\[
        	\|\nabla f(\tilde z)\|_2 \le \|\nabla f_\mu(\tilde z)\|_2 + \mu \|\tilde z - x_0\|_2 \le \frac{\e}{2} + \frac{\e}{2} = \e.
        	\]
        	
            It remains to upper bound the total number of steps of Algorithm~\ref{alg:MSN}. Denote \mbox{$\tilde{c} = \big(2c_{\po{p}} 2^{\frac{p+1}{2}}\big)^{\frac{2}{3p+1}}$} and aggregate~$\tilde N_k$'s from \eqref{eq:tilde N_k for Delta}
            \begin{equation*}
                \sum_{i=0}^k \tilde N_i \leq \tilde{c} \frac{M_p^{\frac{2}{3p+1}} }{\mu^{\frac{p+1}{3p+1}}} \sum_{i=0}^k (\Delta_0\cdot 2^{-i})^{\frac{p-1}{3p+1}} + k \leq \tilde{c} \frac{M_p^{\frac{2}{3p+1}} }{\mu^{\frac{p+1}{3p+1}}} \Delta_0^{\frac{p-1}{3p+1}} \cdot \sum_{i=0}^k  2^{-i\frac{p-1}{3p+1}} + k \notag 
            \end{equation*}
            Since~$p \ge 2$,~$\sum_{i=0}^k  2^{-i\frac{p-1}{3p+1}}$ is a geometric progression with a common ratio of less than one. Therefore, we can upper bound its partial sum by its infinite sum: 
            \begin{equation}\label{eq:geom progression upper bound}
                \sum_{i=0}^k  2^{-i\frac{p-1}{3p+1}} \le \frac{2}{1 - 2^{-\frac{p-1}{3p+1}}} \le 2 \cdot 16 = 32.
            \end{equation}
            Thus, we get
            \begin{gather}
                \sum_{i=0}^k \tilde N_i \leq \tilde{c} \frac{M_p^{\frac{2}{3p+1}} }{\mu^{\frac{p+1}{3p+1}}} \Delta_0^{\frac{p-1}{3p+1}} \cdot \sum_{i=0}^k  2^{-i\frac{p-1}{3p+1}} + k  \leq 32 \tilde{c} \frac{M_p^{\frac{2}{3p+1}} }{\mu^{\frac{p+1}{3p+1}}} \Delta_0^{\frac{p-1}{3p+1}} + \log_2\frac{\Delta_0}{\tilde{\e}} \notag \\
                = O\Bigg( \frac{M_p^\frac{2}{3p + 1}}{\e^\frac{2(p + 1)}{3p + 1}} \Delta_0^\frac{2p}{3p + 1} + \log_2 \frac{2^\frac{4p - 3}{p + 1} \Delta_0 (pM_p)^\frac{1}{p}(p + 1)!}{\e^\frac{p}{p+1}} \Bigg). \label{eq:upper bound on Delta}
            \end{gather}
            
            According to Theorem~\ref{Th:main-conv}, the total number of oracle calls is within the~$O\left(\ln({1}/{\e})\right)$ factor from the number of iterations of Algorithm~\ref{alg:MSN}. This completes the proof.
        \end{proof}
        
        If we omit the dominated factors in the result \eqref{eq:upper bound on Delta}, we obtain the complexity bound 
        \[
            \tilde O  \left( \frac{M_p \Delta_0^p}{\e^{p + 1}} \right)^\frac{2}{3p + 1},
        \]
        where the~$\tilde O$ notation hides an additional multiplicative logarithmic factor. We can conclude that this bound coincides with the lower bound $$\Omega \Big( \frac{M_p \Delta_0^p}{\e^{p + 1}} \Big)^\frac{2}{3p + 1},$$ from~\cite{grapiglia2020tensor} up to logarithmic and constant factors.
        
  \subsection{Near-optimal tensor methods with respect to the initial variable residual}\label{sec:argument}
    
        In this subsection, we build up from Algorithm~\ref{alg:MSN}
        to develop a near-optimal algorithm, we provide explicit complexity bounds for approximating a stationary point. The obtained oracle complexity bound matches the lower bound presented in~\cite{grapiglia2020tensor} up to a logarithmic factor. The basic assumption is that the starting point~$x_0$ satisfies ~$\|x_0 - x^*\|_2 \le R$.
        
        \begin{algorithm}
        	\caption{Near-optimal algorithm for initial argument residual}\label{alg:R}
        	\begin{algorithmic}[1]
        		\REQUIRE $p \ge 2$, $M_p$,~$x_0$,~$R: \|x_0 - x^*\|_2 \le R$,~$\e>0$.
        		
        		\STATE \textbf{Define}:
                \begin{gather*}
                    k = 0, \quad M_\mu = (p + 2)M_p, \quad \mu = \frac{\e}{4R}, \quad \tilde \e = \frac{(\e/2)^{\frac{p+1}{p}}}{4 (p+2)! M_\mu^{\frac{1}{p}}}, \\
                    z_0 = x_0, \quad f_\mu(x) = f(x) + \frac{\mu}{2} \|x - x_0\|_2^2.
                \end{gather*}
                
        		\WHILE {$\mu R_k^2 / 2 \ge \tilde \e$, where~$R_k = R \cdot 2^{-k}$}
        		
        %
        		\STATE~Set~$z_{k + 1} = y_{N_k}$ as the output of Algorithm~\ref{alg:MSN} applied to~$f_{\mu}(x)$ starting from~$z_k$ and run for~$N_k$ steps, where~$N_k$ is such that~$A_{N_k} \ge {4}/{\mu}$.
        		\STATE~$k = k + 1$.
        		\ENDWHILE
        		\STATE \textbf{Find}~$\tilde z = T_{p,\ M_\mu}^{f_\mu}(z_k)$.
        		\RETURN~$\tilde z$.
        		
        	\end{algorithmic}
        \end{algorithm}
        
        \begin{theorem}\label{thm:arg residual}
        	Let~$p \ge 2$. 
            Assume the function~$f$ is convex,~$p$ times differentiable on~$\R^n$ with~$M_p$-Lipschitz~$p$-th derivative. 
            \po{Assume that there exists $R > 0$ is such that  $\|x_0 - x^*\|_2 \le R$.}
            Let~$\tilde z$ be generated by Algorithm~\ref{alg:R}. 
            Then
        	\begin{equation}\label{eq:grad condition}
        	\|\nabla f(\tilde z)\|_2 \le \e
        	\end{equation}
        	and the total number of iterations of Algorithm~\ref{alg:MSN} required by Algorithm~\ref{alg:R} is 
        	\[
        	O \Bigg( \frac{M_p^\frac{2}{3p + 1} R^\frac{2p}{3p + 1}}{\e^\frac{2}{3p + 1}} + \log \frac{2^\frac{p}{p + 1} (p + 1)! (pM_p)^\frac{1}{p}}{\e^\frac{1}{p + 1}} \Bigg).
        	\] 
        	Moreover, the total oracle complexity is within a~$O\left(\ln({1}/{\e})\right)$ factor of the above iteration complexity.
        \end{theorem}
        
        \begin{proof}
        	
        	By definition of~$f_\mu(x)$, we have
        	\begin{equation}\label{eq:initial distance lower R}
        	    f(x_\mu^*) + \frac{\mu}{2} \|x_\mu^* - x_0\|_2^2 = f_\mu(x_\mu^*) \le f_\mu(x^*) = f(x^*) + \frac{\mu}{2}\|x^* - x_0\|_2^2 \leq f(x_\mu^*) + \frac{\mu}{2}\|x^* - x_0\|_2^2.
        	\end{equation}
        	Hence,~$\|x_\mu^* - x_0\|_2^2 \le \|x^* - x_0\|_2^2 \le R^2$. So, for~$k = 0$ we have~$\|x_\mu^* - z_k\|_2 \le R_k$.
        	
        	Let us assume that~$\|x_\mu^* - z_k\|_2 \le R_k$ and show that~$\|x_\mu^* - z_{k + 1}\|_2 \le R_{k + 1}$.
            Again, just like in Algorithm~\ref{alg:Delta}, we use Algorithm~\ref{alg:MSN} inside Algorithm~\ref{alg:R} and restart it every time~$A_{N_k} \ge \frac{4}{\mu}$. We can do this due to the strong convexity of~$f_\mu$.
            From strong convexity, \eqref{eq:residual upper bound by A_N} and this restart condition, it holds that
            \begin{equation}\label{eq:residual R upper bound from the stopping criterion}
                \frac{\mu}{2} \|z_{k + 1} - x_\mu^*\|_2^2 \le f_\mu(z_{k + 1}) - f_\mu(x_\mu^*) \overset{\eqref{eq:residual upper bound by A_N}}{\le} \frac{\|z_k - x^*_\mu\|^2_2}{2 A_{N_k}} \le \frac{R_k^2}{2 A_{N_k}}  \le \frac{\mu R_{k}^2}{8} = \frac{\mu R_{k + 1}^2}{2}.
            \end{equation}
        	Thus,~$\|z_{k + 1} - x_\mu^*\|_2 \le R_{k + 1},\ f_\mu(z_k) - f_\mu(x_\mu^*) \le \frac{\mu R_k^2}{2}$ for all~$k \ge 0$. 

            Since~$f_\mu(x)$ is strongly convex and~$\forall k \ge 0 \Rightarrow \|z_{k} - x_\mu^*\|_2 \le R_k$, we can apply restarts technique.
            In the same way, as in the previous subsection, we can theoretically estimate the upper bound $\tilde N_k$ on the number of  iterations for Algorithm~\ref{alg:MSN} before the stopping criterion~$A_{\tilde N_k} \ge {4}/{\mu}$ holds:
            \begin{equation*}
                f_\mu(z_{k + 1}) - f_\mu(x_\mu^*) 
                \overset{\eqref{eq:residual R upper bound from the stopping criterion}}{\le} 
                \frac{R_k^2}{2 A_{\tilde N_k}}.
            \end{equation*}
            Therefore, if we choose
            \begin{equation}\label{eq:tilde N_k for R}
                \tilde N_k = \max \Bigg\{ \Bigg\lceil \Bigg( \frac{8 c_{\po{p}} M_p R_k^{p - 1}}{\mu} \Bigg)^\frac{2}{3p + 1} \Bigg\rceil, 1 \Bigg\},
            \end{equation}
            then from \eqref{eq:A_k upper bound} we see that this number of steps is sufficient to fulfill the stopping criterion~$A_{\tilde N_k} \ge {4}/{\mu}$:
            \begin{equation*}
                f_\mu(z_{k + 1}) - f_\mu(x_\mu^*) \le \frac{R_k^2}{2 A_{\tilde N_k}} \le \frac{c M_p R_k^{p + 1}}{\tilde N_k^\frac{3p + 1}{2}} \le \frac{\mu R_{k + 1}^2}{2}.
            \end{equation*}
            
        	Next, we estimate~$\|\nabla f_\mu(\tilde z)\|_2$. \po{Since $\tilde z = T_{p, M_\mu}^{f_\mu}(z_k)$, and according to {Lemma~\ref{lem:grap_nest}}, we have}
        	\begin{equation}\label{eq:lemma 5.2 with M_mu}
                f_\mu(z_k) - f_\mu(\tilde z) \ge \frac{1}{4 (p + 2)! M_\mu^\frac{1}{p}} \|\nabla f_{\po{\mu}}(\tilde z)\|_2^\frac{p + 1}{p}
        	\end{equation}
        	At the same time,
        	\begin{equation}\label{eq:R: after outer cycle function residual}
        	   f_\mu(z_k) - f_\mu(\tilde z) \le f_\mu(z_k) - f_\mu(x_\mu^*) \le \frac{\mu R_k^2}{2} \le \tilde \e
        	\end{equation}
        	by the stopping criterion of the algorithm. By combining \eqref{eq:lemma 5.2 with M_mu} with \eqref{eq:R: after outer cycle function residual} and from the choice of~$\tilde \e$ we get that
        	\[
        	   \|\nabla f_\mu(\tilde z)\|_2 \le \frac{\e}{2}.
        	\]
            \po{
            Since the right-hand side of \eqref{eq:lemma 5.2 with M_mu} is non-negative, we can state that 
            \begin{equation}\label{eq:aux2 again}
                f_\mu(\tilde z) \le f_\mu(z_k).
            \end{equation}
            From this and the definition of a strongly convex function, we have that
            }
        	\[
        	\frac{\mu}{2}\|\tilde z - x_\mu^*\|_2^2 \le f_\mu(\tilde z) - f_\mu(x_\mu^*) \overset{\po{\eqref{eq:aux2 again}} }{\le} f_\mu(z_k) - f_\mu(x_\mu^*) \le \frac{\mu R_k^2}{2} = \frac{\mu}{2}(R \cdot 2^{-k})^2 \le \frac{\mu R^2}{2}.
        	\]
        	Thus,~$\|\tilde z - x_\mu^*\|_2 \le R$. Hence,~$\|\tilde z - x_0\|_2 \le \|\tilde z - x_\mu^*\|_2 + \|x_\mu^* - x_0\|_2 \le 2R$.
        	
        	Finally, from our choice of~$\mu$
        	\begin{equation}
        	\|\nabla f(\tilde z)\|_2 \le \|\nabla f_\mu(\tilde z)\|_2 + \mu \|\tilde z - x_0\|_2 \le \frac{\e}{2} + \mu \cdot 2R \le \e.
        	\end{equation}
        	
        	It remains to estimate the upper bound of the number of iterations of the Algorithm~\ref{alg:MSN}. Summing up the number of operations~$\tilde N_i,\ i=0,...,k$ from \eqref{eq:tilde N_k for R}, we obtain
        	\begin{equation*}
        	\sum_{i = 0}^k \tilde N_i \le \sum_{i = 0}^k \Bigg[ \bigg( \frac{8 c_{\po{p}} M_p R_i^{p - 1}}{\mu} \bigg)^\frac{2}{3 p + 1} + 1 \Bigg] = \bigg( \frac{8 c_{\po{p}} M_p R^{p - 1}}{\mu} \bigg)^\frac{2}{3 p + 1} \sum_{i = 0}^k 2^\frac{-2i (p - 1)}{3p + 1} + k 
            \end{equation*}
            Again, as in Theorem~\ref{thm:func residual}, since~$p \ge 2$,~$\sum_{i = 0}^k 2^\frac{-2i (p - 1)}{3p + 1}$ is a geometric progression with a common ratio lower than one. Therefore, we can upper bound its partial sum with its infinite sum.
            \[
                \sum_{i = 0}^k 2^\frac{-2i (p - 1)}{3p + 1} \le \frac{2}{1 - 2^\frac{-2(p - 1)}{3p + 1}} \le 2 \cdot 5 = 10.
            \]
            \begin{gather}
        	   \sum_{i = 0}^k \tilde N_i \leq \bigg( \frac{8 c_{\po{p}} M_p R^{p - 1}}{\mu} \bigg)^\frac{2}{3 p + 1} \sum_{i = 0}^k 2^\frac{-2i (p - 1)}{3p + 1} + k  \le 10\bigg( \frac{8 c_{\po{p}} M_p R^{p - 1}}{\mu} \bigg)^\frac{2}{3 p + 1} + \frac{1}{2}\log_2\frac{\mu R^2}{2 \tilde{\e}} \notag \\
        	   = O \Bigg( \frac{M_p^\frac{2}{3p + 1} R^\frac{2p}{3p + 1}}{\e^\frac{2}{3p + 1}} + \frac{1}{2}\log \frac{2^\frac{p}{p - 1} (p + 1)! M_p^\frac{1}{p}}{\e^\frac{1}{p + 1}} \Bigg). \label{eq:upper bound on R}
        	\end{gather}
            According to Theorem~\ref{Th:main-conv}, the total number of oracle calls is within the~$O\left(\ln({1}/{\e})\right)$ factor from the number of iterations of Algorithm~\ref{alg:MSN}.  
        	This completes the proof.
        \end{proof}
        
        If we omit the dominated factors in the result \eqref{eq:upper bound on R}, we obtain the complexity bound
        \begin{equation}
        \label{eq:arg_compl_simple}
            \tilde O \left( \frac{M_p R^p}{\e} \right)^\frac{2}{3p + 1}.
        \end{equation}
        Therefore, we can conclude that this bound coincides with the lower bound $$\Omega \left( \frac{M_p R^p}{\e}  \right)^\frac{2}{3p + 1},$$ from~\cite{grapiglia2020tensor} up to logarithmic and constant factors.
        
        \begin{remark}\label{rem:extension to strongly convex case}
        As a byproduct, Algorithm~\ref{alg:R} can minimize functions~$f$ that are already strongly convex. Indeed, in this case, we deal with the objective~$f$ as we now deal with the auxiliary objective~$f_{\mu}$: we apply Algorithm~\ref{alg:MSN} by epochs to~$f$ and restart when the stopping criterion holds. 
        \po{Since in this case $\mu$ is not a regularization coefficient, but just a constant of strong convexity, we do not need to subsitute $\mu$ with $\e^2 / (32 \Delta_0)$ in \eqref{eq:upper bound on Delta} and with $\e / (4R)$ in \eqref{eq:upper bound on R}. 
        Thus, we get the following complexity estimations:}
        \begin{align}\label{eq:gasnikov strongly convex estimation}
            \tilde O \left( \frac{M_p \Delta_0^\frac{p - 1}{2}}{\mu^\frac{p + 1}{2}} \right)^\frac{2}{3p + 1} \quad \text{and} \quad
        \tilde O \left(  \frac{M_p R^{p - 1}}{\mu} \right)^\frac{2}{3 p + 1}.
        \end{align}
        The main difference between Algorithms~\ref{alg:Delta} and~\ref{alg:R} is that in both cases, we use~\eqref{eq:A_k upper bound} to estimate the number of inner iterations of Algorithm~\ref{alg:MSN}, but in the first case we additionally use strong convexity to be able to upper bound argument residual with functional residual in \eqref{eq:A_k upper bound}: $R_k \leq \sqrt{({2}/{\mu}) \Delta_k}$.
        \end{remark}
        
        \begin{remark}
        \po{Let us now derive a complexity estimation for finding approximate solution to \eqref{eq:PrSt}, using Algorithm \ref{alg:R}.}
        The idea is to apply Algorithm~\ref{alg:R} to the dual problem and then use Proposition~\ref{prop:dual gradient norm and primal solution}.
        To that end, we set the following equivalence between the notation of Section~\ref{sec:motivation} and the notation of this section
       ~$\lambda \equiv z$,~$\vp(\lambda) \equiv f(z)$. We start by estimating the number of iterations of Algorithm~\ref{alg:MSN} to fulfill the first condition of Proposition~\ref{prop:dual gradient norm and primal solution}
        \begin{equation}\label{eq:dual gradient norm and primal solution, 1st cond}
            -\la \lambda_k, \nabla \vp(\lambda_k) \ra \le \e_f.
        \end{equation}
        Assume that after applying Algorithm~\ref{alg:R}, we obtain a point~$\lambda_k$ such that~\mbox{$\|\nabla \vp(\lambda_k)\|_2 \le \e$.} Then,
        \begin{equation}\label{eq:dual problem cauchi-schwarz}
            - \la \lambda_k, \nabla \vp(\lambda_k) \ra \le \|\lambda_k\|_2 \|\nabla \vp(\lambda_k)\|_2 \le \e \|\lambda_k\|_2
        \end{equation}
        From the triangle inequality, \eqref{eq:initial distance lower R}, and \eqref{eq:residual R upper bound from the stopping criterion}, we have
        \[
        \|\lambda_k\|_2 \leq \|\lambda_0\|_2 + \|\lambda_0 - \lambda_\mu^*\|_2  + \|\lambda_k - \lambda_\mu^*\|_2 
        \overset{\eqref{eq:initial distance lower R},\eqref{eq:residual R upper bound from the stopping criterion}}{\le} \|\lambda_0\|_2 + 2R,
        \]
        where we also used that from \eqref{eq:residual R upper bound from the stopping criterion}~$\|\lambda_k - \lambda_\mu^*\|_2 \leq R_k \leq R$.
        Since~$\lambda_0$ is our choice (in particular, we can start our algorithm from~$\lambda_0 = 0$), we can use this inequality to estimate~$\|\lambda_k\|_2$. 
        From the above and \eqref{eq:dual problem cauchi-schwarz}, we get
        \[
            - \la \lambda_k, \nabla \vp(\lambda_k) \ra \le \e (2R + \|\lambda_0\|_2) = \e_f. 
        \]
        Thus, if we set~$\e = \frac{\e_f}{2R + \|\lambda_0\|_2}$, we obtain that the first condition of Proposition~\ref{prop:dual gradient norm and primal solution} holds.
        If we set~$\e=\e_{eq}$, we also obtain the second condition of this proposition. 
        Setting~$\lambda_0=0$ and~$\e = \min\{\frac{\e_f}{2R},\e_{eq} \}$, and applying the bound \eqref{eq:arg_compl_simple}, we finally obtain the following complexity bound for finding an approximate solution to problem \eqref{eq:PrStGen} in the sense of \eqref{eq:e_f_e_eq-solution}
        \[
        \tilde{O} \left(\max \left\{  \left( \frac{M_p R^{p + 1}}{\e_f} \right)^\frac{2}{3p + 1}  , \left( \frac{M_p R^{p }}{\e_{eq}} \right)^\frac{2}{3p + 1} \right\}\right).
        \]
        
        
        \end{remark}
        
        
        While Algorithms~\ref{alg:Delta} and~\ref{alg:R} are shown to be near-optimal, the price of optimality of the algorithm is high. We believe that pointing out this price of optimality can lead to future research in computationally tractable approaches. Specifically, the algorithms require a line search process, which adds logarithmic terms to the complexity. Moreover, they depend on restart techniques and regularization, whose parameters depend on the desired accuracy~$\varepsilon$ and other parameters such as~$R$, which are assumed to be known.
        

\section{Primal-dual accelerated tensor method}\label{sec:primal_dual}
 
    In Section~\ref{sec:algos}, we considered methods, which search for an approximate stationary point of the dual problem, and then reconstruct an approximate solution to the primal problem. 
    Another approach to tackle \eqref{eq:PrStGen} is via primal-dual methods. The main idea of these methods is to solve both dual and primal problems until both duality gap~$|f(x(\lambda_k)) + \vp(\lambda_k)|$ and equality constraint residual of primal variable~$\|A x(\lambda_k) - b\|_2$ are lower than some accuracy~$\e$.
    
    In this section, we compare these two approaches theoretically and numerically. Hence, in this section, we propose an accelerated primal-dual tensor method (Algorithm~\ref{Alg:PDATM}) and provide its theoretical comparison with Algorithms~\ref{alg:Delta} and~\ref{alg:R} in Remark~\ref{rem:pd and gd comparison}. 
    Our proposed method uses the framework of estimating sequences~\cite{nesterov2004introduction}, where in each step it solves high-order optimization Problem~\eqref{eq:tensor step} for the dual function~$\vp$.
   
    
    First, recall formulation of the dual problem for \eqref{eq:PrStGen}
     \begin{equation}\label{eq:DualPr_again}
        \min_{\lambda \in H^*} \left\{ \vp(\lambda) := \la \lambda, b \ra + \max_{x\in Q} \left(- f(x) - \la A^\intercal \lambda,x \ra \right) \right\}.
    \end{equation}
    We have already mentioned it in Section~\ref{sec:motivation}. 
    \po{We assume the dual function $\varphi$ has $M_{p}$-Lipschitz $p$-th order derivative (see, e.g., Section \ref{sec:motivation}).}
    From the weak duality, the following inequality follows
    \begin{equation}
        f(x^*) \geq -\vp(\lambda^*),
    \label{eq:wD}
    \end{equation}
    where~$f(x^*)$ and~$\vp(\lambda^*)$ are the optimal function values in~\eqref{eq:PrStGen} and~\eqref{eq:DualPr} respectively.

    Assume the dual problem \eqref{eq:DualPr_again} has a solution~$\lambda^*$ (which holds, e.g., when the strong duality holds), and there exists some~$R >0$ such that
    	\begin{equation}
    	\|\lambda^{*}\|_{2} \leq R < +\infty. 
    	\label{eq:l_bound}
    	\end{equation}	 
    It is worth noting that the quantity~$R$ will be used only in the convergence analysis but not in the algorithm itself. 

    To solve the dual Problem~\eqref{eq:DualPr_again}, we introduce the Primal-Dual Accelerated Tensor Method (Algorithm~\ref{Alg:PDATM}).
    \begin{algorithm}[t]
    \caption{Primal-Dual Accelerated Tensor Method}
    \label{Alg:PDATM}
    {\small
    \begin{algorithmic}[1]
    		\REQUIRE~$\e_f$,~$\e_{eq}$,~$M > M_p$.
    
    		\STATE Set~$k=0$, ~$\lambda_0 = 0$,~$\psi_{0}(\lambda) =  \frac{C}{(p+1)!} \norm{\lambda - \lambda_{0}}_2^{p+1}$, where ~$C = \tfrac{p}{2}\sqrt{\tfrac{p+1}{p-1}(M^2 - M_p^2)}$.
    		\REPEAT
    			\STATE Compute~$v_k = \argmin\limits_{\lambda }\psi_k(\lambda).$
    			\STATE~$A_{k} = \left[\tfrac{(p-1)(M^2-M_p^2)}{4(p+1)p^2M^2} \right]^{\frac{p}{2}}\left(\tfrac{k}{p+1} \right)^{p+1}$,~$a_k = A_{k+1} - A_{k}$.
    			
    			\STATE~$y_k = \frac{A_k}{A_{k+1}}\lambda_k + \frac{a_k}{A_{k+1}}v_k$.
    			
    			\STATE Compute~$\lambda_{k+1} = T^{\vp}_{p,M}(y_k)$.
    			
                \STATE
    			\begin{equation*}
    			\psi_{k+1}(\lambda) = \psi_{k}(\lambda) + (A_{k+1}-A_{k})\left[\vp(\lambda_{k+1}) + \left<\nabla \vp(\lambda_{k+1}), \lambda - \lambda_{k+1} \right> \right].
    			\end{equation*}

    			\STATE
    			\begin{equation*}
    			\hat x_{k+1} = \frac{1}{A_{k+1}}\sum_{i=0}^{k}a_{i}x(\lambda_{i+1})=\frac{a_{k}x(\lambda_{k+1}) + A_{k}\hat x_{k}}{ A_{k+1}}
    			\end{equation*}
    			
    			\STATE Set~$k = k + 1$.
    			
    		\UNTIL {$|f(\hat x_{k}) + \vp(\lambda_{k})| \leq \e_f$,~$\|A \hat{x}_k - b \|_2 \leq \e_{eq}$}.
    		\RETURN~$\hat x_{k}$,~$\lambda_{k}$.	
    \end{algorithmic}
    }
    \end{algorithm}
     To prove the main result about the convergence of Algorithm~\ref{Alg:PDATM}, we need the following auxiliary lemmas. 
     
    \begin{lemma}[Corollary 1 in ~\cite{nesterov2021implementable}]
    \label{lem_4}
    For any~$\lambda \in H^*$ and~$M \geq  M_p$ we have 
    \begin{eqnarray}
    \label{eq_lem_2}
        \langle \nabla\varphi(T^{\varphi}_{p,M}(\lambda)), \lambda - T^{\varphi}_{p,M}(\lambda) \rangle \geq \tfrac{c(p)}{M}[M^2 - M_p^2]^{\frac{p-1}{2p}}\|\nabla\varphi(T^{\varphi}_{p,M}(\lambda))\|_{2}^{\frac{p+1}{p}},
    \end{eqnarray}
    where~$c(p) = \tfrac{p}{p-1}\left[\tfrac{p-1}{p+1}\right]^{\frac{1-p}{2p}}[(p+1)!]^{\frac{1}{p}}$.
    \end{lemma} 

    \begin{lemma}[Lemma 2 in~\cite{nesterov2008accelerating}]
        Let~$\sigma > 0$ be some constant. Then, for any~$h \in E$ and~$s \in E$, we have
        \begin{equation}\label{eq:pd:estimating sequences last step inequality}
            \la s, h \ra + \frac{1}{p} \sigma \|h\|_2^p \ge -\frac{p - 1}{p} \left( \frac{1}{\sigma} \right)^\frac{1}{p - 1} \|s\|_2^\frac{p}{p - 1}.
        \end{equation}
    \end{lemma}
    
    Let us introduce the following estimating functions, which are recursively updated as 
    \begin{eqnarray}
    \label{psi_rec}
        \forall k \ge 0 \Rightarrow \psi_{k+1}(\lambda) =  \psi_{k}(\lambda) + a_{k}\left[\vp(\lambda_{k+1}) + \left<\nabla \vp(\lambda_{k+1}), \lambda - \lambda_{k+1} \right> \right]
    \end{eqnarray}
    with~$\psi_{0}(\lambda) =  \frac{C}{(p+1)!} \norm{\lambda - \lambda_{0}}_{2}^{p+1}$, where~$C = \tfrac{p}{2}\sqrt{\tfrac{p+1}{p-1}(M^2 - M_p^2)}$.

    \begin{theorem}
    \label{th1}
    If sequence~$\{\lambda_k\}_{k=0}^{\infty}$ is generated by Algorithm~\ref{Alg:PDATM}, then for all~$k \geq 0$ we have 
    \begin{equation}
    \label{varphi_estimation}
        A_{k}\vp(\lambda_{k}) \leq \min\limits_{\lambda \in H^*} \psi_k(\lambda).
    \end{equation}
    \end{theorem}
    
    \begin{proof}
        Let us prove the relation \eqref{varphi_estimation} by induction over~$k$. Since~$A_{0} = 0$, for~$k=0$ we obtain:
       ~$$
        0 = A_0\vp(\lambda_{0})\leq \min\limits_{\lambda \in H^*}\frac{C}{(p+1)!} \norm{\lambda - \lambda_{0}}_{2}^{p+1} = 0.
       ~$$
        Assume that \eqref{varphi_estimation} is true for some~$k > 0$. Denote
       ~$$
            \psi_{k}(\lambda) \equiv l_k(\lambda) + \frac{C}{(p+1)!} \norm{\lambda - \lambda_{0}}_{2}^{p+1} \quad k \geq 0,
       ~$$
        where~$l_0(\lambda) \equiv 0$.
        
        Using Lemma 4 in~\cite{nesterov2008accelerating} we obtain that
        \begin{eqnarray*}
            \psi_{k}(\lambda) &\geq& \min\limits_{\lambda \in H^*}  \psi_{k}(\lambda) + \tfrac{C}{(p+1)!}\cdot \left(\tfrac{1}{2} \right)^{p-1} \norm{\lambda - v_k}_{2}^{p+1} \\
            & \geq & A_{k}\vp(\lambda_{k}) + \tfrac{C}{(p+1)!}\cdot \left(\tfrac{1}{2} \right)^{p-1} \norm{\lambda - v_k}_{2}^{p+1}.
        \end{eqnarray*}
        Denote~$\sigma_{p + 1} = \frac{C}{p!} \left(\frac{1}{2}\right)^{p - 1}$. Then, from this inequality and~$\vp(\lambda_k) - \vp(\lambda_{k + 1}) \ge \la \nabla \vp(\lambda_{k + 1}, \lambda_k - \lambda_{k + 1} \ra$, we get
        \begin{eqnarray}
        \label{eq:th_pd}
            && \psi_{k+1}^* =  \min\limits_{\lambda \in H^*} \left\{ \psi_{k}(\lambda) + a_k[\vp(\lambda_{k+1}) + \langle \nabla \vp(\lambda_{k+1}), \lambda - \lambda_{k+1} \rangle ] \right\} \notag \\ &\geq&   \min\limits_{\lambda \in H^*} \Bigl\{ A_{k}\vp(\lambda_{k}) + \tfrac{\sigma_{p+1}}{(p+1)} \norm{\lambda - v_k}_{2}^{p+1} + a_k[\vp(\lambda_{k+1}) + \langle \nabla \vp(\lambda_{k+1}), \lambda - \lambda_{k+1} \rangle ] \Bigr\}  \notag  \\ &\geq &  \min\limits_{\lambda \in H^*} \Bigl\{ A_{k+1} \vp(\lambda_{k+1}) + A_k \langle \nabla \vp(\lambda_{k+1}), \lambda_{k} - \lambda_{k+1}  \rangle  \notag  \\ & + &  a_k \langle \nabla \vp(\lambda_{k+1}), \lambda - \lambda_{k+1}  \rangle  +  \tfrac{\sigma_{p+1}}{(p+1)} \norm{\lambda - v_k}_{2}^{p+1} \Bigr\}  
        \end{eqnarray}
        Note, that~$y_{k}=\frac{A_{k}}{A_{k+1}}\lambda_{k}+\frac{a_{k}}{A_{k+1}}v_{k}$. Hence,~$A_{k}\lambda_{k}=A_{k+1}y_{k}-a_{k}v_{k}$, and 
        \begin{equation*}
            A_{k}\left<\nabla \vp(\lambda_{k+1}), \lambda_k - \lambda_{k+1} \right> =  \left<\nabla \vp(\lambda_{k+1}), A_{k+1}y_{k}-a_{k}v_{k} - A_{k} \lambda_{k+1} \right>.
        \end{equation*}
        Thus, we can rewrite~\eqref{eq:th_pd} as follows
        \begin{eqnarray}
             \psi_{k+1}^*  &\geq&   \min\limits_{\lambda \in H^*} \Bigl\{ A_{k+1} \vp(\lambda_{k+1}) +  \left<\nabla \vp(\lambda_{k+1}), A_{k+1}y_{k}-a_{k}v_{k} - A_{k} \lambda_{k+1} \right>   \notag \\ 
             & + &  a_k \langle \nabla \vp(\lambda_{k+1}), \lambda - \lambda_{k+1}  \rangle  +  \tfrac{\sigma_{p+1}}{(p+1)} \norm{\lambda - v_k}_{2}^{p+1} \Bigr\}  \notag \\ 
             & = & \min\limits_{\lambda \in H^*} \Bigl\{ A_{k+1} \vp(\lambda_{k+1}) +  A_{k+1} \left<\nabla \vp(\lambda_{k+1}), y_{k}- \lambda_{k+1} \right> \notag \\ 
             & + &  a_k \langle \nabla \vp(\lambda_{k+1}), \lambda - v_{k}  \rangle  +  \tfrac{\sigma_{p+1}}{(p+1)} \norm{\lambda - v_k}_{2}^{p+1} \Bigr\} \label{eq:pd:th_pd_2}
        \end{eqnarray}
        Further, if we choose~$M \geq M_{p}$, then by inequality~\eqref{eq_lem_2} we have 
        \begin{equation}\label{eq:pd:gradient on distance lower bound}
            \langle \nabla \varphi(\lambda_{k+1}), \lambda - \lambda_{k+1} \rangle \geq \tfrac{c(p)}{M}[M^2 - M_p^2]^{\frac{p-1}{2p}}\|\nabla \varphi(\lambda_{k+1})\|_{2}^{\frac{p+1}{p}}.
        \end{equation}
        If we apply \eqref{eq:pd:gradient on distance lower bound} to \eqref{eq:pd:th_pd_2}, we get
        \begin{align*}
             \psi_{k + 1}^* \ge \min_{\lambda \in H^*} \Big\{ A_{k + 1} \vp(\lambda_{k + 1}) &+ A_{k+1}\tfrac{c(p)}{M}[M^2 - M_p^2]^{\frac{p-1}{2p}}\|\nabla \varphi(\lambda_{k+1})\|_{2}^{\frac{p+1}{p}} \\
             &+ a_k \langle \nabla \vp(\lambda_{k+1}), \lambda - v_{k}  \rangle +  \tfrac{\sigma_{p+1}}{(p+1)} \norm{\lambda - v_k}_{2}^{p+1} \Big\}.
        \end{align*}
        Now, denote everything on the right-hand side except~$A_{k + 1} \vp(\lambda_{k + 1})$ as~$\zeta(\lambda)$:
        \[
            \zeta(\lambda) \equiv A_{k+1}\tfrac{c(p)}{M}[M^2 - M_p^2]^{\frac{p-1}{2p}}\|\nabla \varphi(\lambda_{k+1})\|_{2}^{\frac{p+1}{p}} + a_k \langle \nabla \vp(\lambda_{k+1}), \lambda - v_{k} \rangle + \frac{\sigma_{p + 1}}{p + 1}\|\lambda - v_k\|_{2}^{p + 1}.
        \]
        Thus,
        \[
            \psi_{k + 1}^* \ge \min_{\lambda \in H^*} \Big\{ A_{k + 1} \vp(\lambda_{k + 1}) + \zeta(\lambda) \Big\}.
        \]
        
        To prove \eqref{varphi_estimation}, we need to have~$\zeta(\lambda) \ge 0$.
        Using \eqref{eq:pd:estimating sequences last step inequality}, we get
        \begin{equation*}
            \zeta(\lambda) \ge 
            A_{k+1}\tfrac{c(p)}{M}[M^2 - M_p^2]^{\frac{p-1}{2p}}\|\nabla \varphi(\lambda_{k+1})\|_{2}^{\frac{p+1}{p}} - \frac{p}{p + 1} \left( \frac{1}{\sigma_{p + 1}} \right)^\frac{1}{p} a_k^\frac{p + 1}{p} \|\nabla \vp(\lambda_{k + 1}) \|_{2}^\frac{p + 1}{p}.
        \end{equation*}
        Therefore, to have~$\zeta(\lambda) \ge 0$ we need
        \begin{gather*}
            A_{k+1}\tfrac{c(p)}{M}[M^2 - M_p^2]^{\frac{p-1}{2p}} \ge \frac{p}{p + 1} \left( \frac{1}{\sigma_{p + 1}} \right)^\frac{1}{p} a_k^\frac{p + 1}{p}.
        \end{gather*}
        Next we substitute in this inequality the values of~$c(p)$ and~$\sigma_{p + 1}$ and after all the constellations we get
        \[
            A_{k + 1} \sqrt{1 - \frac{M_p^2}{M^2}} \left(\frac{C^2}{M^2 - M_p^2} \right)^\frac{1}{2p} \ge 2 a_k^\frac{p + 1}{p} \left( \frac{p}{2} \sqrt{\frac{p + 1}{p - 1}} \right)^\frac{1}{p} \sqrt{\frac{p + 1}{p - 1}}.
        \]
        Finally, from our choice of~$C$, we get
        \begin{eqnarray}
        \label{eq_estimation}
             A_{k+1} \ge 2 \sqrt{\tfrac{(p+1)M^2}{(p-1)(M^2-M_p^2)}}a_k^{\frac{p+1}{p}}.
        \end{eqnarray}
        And since for~$k \geq 0$
        \begin{eqnarray*}
             A_{k} = \left[\tfrac{(p-1)(M^2-M_p^2)}{4(p+1) M^2} \right]^{\frac{p}{2}}\left(\tfrac{k}{p+1} \right)^{p+1}, \quad a_k = A_{k+1} - A_{k}. 
        \end{eqnarray*}
        inequality \eqref{eq_estimation} holds. It is described in more detail in~\cite{nesterov2021implementable} (everything from eq. (3.8) to eq. (3.11)).
        Eventually, 
        \[
            \psi_{k + 1}^* \ge \min_{\lambda \in H^*} \Big\{ A_{k + 1} \vp(\lambda_{k + 1}) + \zeta(\lambda) \Big\} \ge A_{k + 1} \vp(\lambda_{k + 1}),
        \]
        that completes the induction argument
    \end{proof}
    
    We can now estimate the proposed algorithm's complexity. Consider the set~$\Lambda_R =\{\lambda: \|\lambda\|_2 \leq 2R\}$ where~$R$ is given in \eqref{eq:l_bound}. From the Theorem~\ref{th1} and since~$\lambda_0 = 0$ we obtain 
    \begin{align}
        A_{k}\vp(\lambda_{k}) &\leq \min\limits_{\lambda}\left\{ \sum_{i = 0}^{k-1} a_{i}\left[\vp(\lambda_{i+1}) + \left<\nabla \vp(\lambda_{i+1}), \lambda - \lambda_{i+1} \right> \right]+ \frac{C}{(p+1)!} \norm{\lambda}^{p+1}_{2}\right\} \notag \\
        &\leq \min\limits_{\lambda \in \Lambda_R}\left\{ \sum_{i = 0}^{k-1} a_{i}\left[\vp(\lambda_{i+1}) + \left<\nabla \vp(\lambda_{i+1}), \lambda - \lambda_{i+1} \right> \right]+ \frac{C}{(p+1)!} \norm{\lambda }^{p+1}_{2}\right\} \notag \\
        & \leq \min\limits_{\lambda \in \Lambda_R}\left\{ \sum_{i = 0}^{k-1} a_{i}\left[\vp(\lambda_{i+1}) + \left<\nabla \vp(\lambda_{i+1}), \lambda - \lambda_{i+1} \right> \right]\right\}+ \frac{C(2R)^{p+1}}{(p+1)!}.
    \label{eq:proof_st_1}
    \end{align}
    On the other hand, from the definition \eqref{eq:DualPr_again} of~$\vp(\lambda)$, we have
    \begin{eqnarray*}
        \vp(\lambda_i) & = &\la \lambda_i, b \ra \notag + \max_{x\in Q} \left( -f(x) - \la A^\intercal \lambda_i ,x \ra \right) 
    \\ & = &\la \lambda_i, b \ra - f(x(\lambda_i)) - \la A^\intercal \lambda_i ,x(\lambda_i) \ra . \notag
    \end{eqnarray*}
    And since~$\nabla \vp(\lambda) = b - Ax(\lambda)$, we obtain
    \begin{align}
        \vp(\lambda_i) - \la \nabla \vp (\lambda_i), \lambda_i \ra & = \la \lambda_i, b \ra  - f(x(\lambda_i)) - \la A^\intercal \lambda_i ,x(\lambda_i) \ra \notag \\
        & \hspace{1em} - \la b-A x(\lambda_i),\lambda_i \ra  = - f(x(\lambda_i)). \notag
    \end{align}
    Summing these inequalities from~$i=0$ to~$i=k-1$ with the weights~$\{\alpha_i\}_{i=0,...k-1}$, we get, using the convexity of~$f$
    \begin{align}
        &  \sum_{i=0}^{k-1}{\alpha_i \left( \vp(\lambda_{i+1}) + \la \nabla \vp(\lambda_{i+1}), \lambda-\lambda_{i+1} \ra \right) }  \notag \\
        & = -\sum_{i=0}^{k-1} \alpha_i f(x(\lambda_{i+1})) + \sum_{i=0}^{k-1} \alpha_i \la b - Ax(\lambda_{i+1}), \lambda \ra  \notag \\
        & \leq -A_{k}f(\hat{x}_k) + A_{k} \la b - A \hat{x}_k, \lambda \ra, \notag
    \end{align}
    where~$\hat x_{k} = \frac{1}{A_{k}}\sum\limits_{i=0}^{k-1}a_{i}x(\lambda_{i+1})$. Substituting this inequality to \eqref{eq:proof_st_1}, we obtain
    \begin{align}
        A_k \vp(\lambda_k)  \leq &-A_{k}f(\hat{x}_k) + A_{k}  \min_{\lambda \in \Lambda_R} \left\{ \la b - A \hat{x}_k, \lambda \ra  \right\} +  \frac{C(2R)^{p+1}}{(p+1)!}. \notag
    \end{align} 
    Finally, since 
    \begin{align}
        &\max_{\lambda \in \Lambda_R} \left\{   \la A \hat{x}_k - b, \lambda \ra  \right\}  =2 R \|A \hat{x}_k - b \|_2 , \notag
    \end{align} 
    we obtain
    \begin{equation}
    \vp(\lambda_k) + f(\hat{x}_k) +2 R \|A \hat{x}_k - b \|_2   \leq \frac{C(2R)^{p+1}}{A_k (p+1)!}.
    \label{eq:vpmfxh}
    \end{equation}
    Since ~$\lambda^*$ is an optimal solution of dual problem ~\eqref{eq:DualPr_again}, we have, for any~$x \in Q$
   ~$$
    f(x^*) \leq f(x) + \la  \lambda^{*}, Ax - b \ra.
   ~$$
    Using the assumption \eqref{eq:l_bound} we get
    \begin{equation}
    f(\hat{x}_k) \geq f(x^*)- R \|A \hat{x}_k - b \|_2.
    \label{eq:fxhat_est}
    \end{equation}
    Hence,
    \begin{align}
     \vp(\lambda_k) + f(\hat{x}_k)  & = \vp(\lambda_k) - \vp(\lambda^*)+\vp(\lambda^*) + f(x^*)  - f(x^*) + f(\hat{x}_k)  \notag \\
    &  \stackrel{\eqref{eq:wD}}{\geq}  - f(x^*) + f(\hat{x}_k) \stackrel{\eqref{eq:fxhat_est}}{\geq} - R \|A \hat{x}_k - b \|_2.
    \label{eq:aux1}
    \end{align}
    This and \eqref{eq:vpmfxh} give
    \begin{equation}
    R \|A \hat{x}_k - b \|_2  \leq \frac{C(2R)^{p+1}}{A_k (p+1)!}.
    \label{eq:R_norm_est}
    \end{equation}
    Hence, we obtain
    \begin{equation}
     -\frac{C(2R)^{p+1}}{A_k (p+1)!} \stackrel{\eqref{eq:aux1},\eqref{eq:R_norm_est}}{\leq} \vp(\lambda_k) + f(\hat{x}_k) \stackrel{\eqref{eq:vpmfxh}}{\leq} \frac{C(2R)^{p+1}}{A_k (p+1)!} .
    \label{eq:vppfxhatgeq}
    \end{equation}
    Combining \eqref{eq:R_norm_est} and  \eqref{eq:vppfxhatgeq}, we conclude
    \begin{eqnarray*}
    R\|A \hat{x}_k - b \|_2 \leq \frac{C(2R)^{p+1}}{A_k (p+1)!}, \quad  |\vp(\lambda_k) + f(\hat{x}_k)|  \leq \frac{C(2R)^{p+1}}{A_k (p+1)!}.
    \end{eqnarray*}
    Finally, if we put the value of~$A_k$, defined in Algorithm~\ref{Alg:PDATM}, we will get the total complexity of the Algorithm~\ref{Alg:PDATM}. Therefore, we have just proved the following theorem.
    
    \begin{theorem}\label{th2}
        Assume the function~$\varphi$ from \eqref{eq:DualPr_again} is convex,~$p$ times differentiable on~$\R^m$ with~$M_p$-Lipschitz~$p$-th derivative. Additionally, if~$\lambda^* = \arg \min_{\lambda \in H^*} \vp(\lambda)$, assume~$\exists R > 0: \| \lambda^* \|_2 \le R \le \infty$. Let Algorithm~\ref{Alg:PDATM} be run for~$k$ steps with starting point~$\lambda_0 = v_0 = z_0 = 0$.
        Denote~$\hat x_{k} = \frac{1}{A_{k}}\sum\limits_{i=0}^{k-1}a_{i}x(\lambda_{i+1})$. Then
        \begin{align*}
            &\|A \hat{x}_k - b \|_2 \leq \dfrac{C_1 R^{p}}{k^{p+1}}, \notag \\
            &|\vp(\lambda_k) + f(\hat{x}_k)| \leq \dfrac{C_1 R^{p+1}}{k^{p+1}} .
        \end{align*}
        Here~$C_1=\tfrac{4^p M^p }{(p-1)!}\sqrt{\tfrac{(p+1)^{3p+3}}{(M^2 - M_p^2)^{p-1}(p-1)^{p+1}}}$.
    \end{theorem}

    The result of the above Theorem can be written in terms of complexity in the following way. Assume that the goal is to find an approximate solution that satisfies inequalities \eqref{eq:e_f_e_eq-solution}. Then, Theorem~\ref{th2} states that such a point can be found in a number of iterations not exceeding
    \[
    O \left(\max \left\{  \left( \frac{M_p R^{p + 1}}{\e_f} \right)^\frac{1}{p + 1}  , \left( \frac{M_p R^{p }}{\e_{eq}} \right)^\frac{1}{p + 1} \right\}\right).
    \]
    It is an open question whether it is possible to obtain a primal-dual tensor method with complexity bounds that depend on~$\e^{-2/(3p+1)}$ rather than~$\e^{-1/(p+1)}$.
    
    \begin{remark}\label{rem:pd and gd comparison}
   \cu{ Let us now discuss Algorithm~\ref{Alg:PDATM} compared to Algorithms~\ref{alg:Delta} and~\ref{alg:R}. On the one hand, complexity~$O(\e^{-1/(p+1)})$ of the former  has \textit{asymptotically} worse dependence on~$\e$ than the complexity bound~$\tilde{O}(\e^{-2/(3p+1)})$ of the latter. On the other hand, the difference in the power of~$\e$ is quite small, and the second bound has an additional logarithmic multiplier. Thus, the bound for Algorithms~\ref{alg:Delta} and~\ref{alg:R} may be only slightly better than the bound for Algorithm~\ref{Alg:PDATM}.
    Further, Algorithm~\ref{Alg:PDATM} is a direct algorithm that does not use regularizations and restarts. Unlike it, Algorithms~\ref{alg:Delta} and~\ref{alg:R} use a regularization that may be so small that it will cause some numerical instabilities. In Section~\ref{sec:numerics}, we compare both approaches numerically. At the same time, Algorithms~\ref{alg:Delta} and~\ref{alg:R} are interesting not only in application to problem \eqref{eq:PrStGen}. These methods achieve nearly-optimal complexity bounds for finding approximate stationary points of convex functions, nearly closing the theoretical gap. Some other motivations for developing efficient methods for finding stationary points can be found in \cite{grapiglia2020tensor}. In particular, the norm of the gradient is a natural and computable measure of optimality. }
    \end{remark}

    \begin{remark}\label{rem:extensions}
    Let us discuss a possible extension of the proposed methods. One straightforward generalization is a near-optimal method for minimizing the norm of objective with H\"older-continuous gradient, i.e.,  for some~$\nu \in [0,1]$ satisfying
    \[
    \|\nabla^p f(x) - \nabla^p f(y)\|_2 \leq M_{p,\nu} \|x-y\|_2^{\nu}, x,y \in \R^n.
    \]
    The idea is to combine the near-optimal tensor method for minimization of functions with H\"older-continuous~$p$-th derivatives~\cite{song2019towards} with Lemma 5.2 in~\cite{grapiglia2020tensor} for general~$\nu$. This approach allows obtaining complexity bounds which, up to logarithmic and constant factors, coincide with the lower bounds in~\cite{grapiglia2020tensor}. 
    
    Another possible extension is an inexact solution of the auxiliary subproblems and adaptation to the constant~$M_{p,\nu}$~\cite{grapiglia2020tensor}. Importantly, the basic Algorithm~\ref{alg:MSN} is adaptive to~$M_p$. Nevertheless, to apply the regularization technique with parameter~$\mu$, we need to know~$M_p$. Thus, it is desirable to overcome this drawback.
    
    Finally, in our Algorithm~\ref{Alg:PDATM} we use  Nesterov acceleration based on the estimating sequence technique (see~\cite{nesterov2004introduction, nesterov2021implementable}). It is still an open question whether we can obtain a better high-order primal-dual method using the Monteiro-Svaiter acceleration~\cite{monteiro2013accelerated} or optimal tensor method~\cite{kovalev2022first}.
    \end{remark}
    
\section{Numerical analysis}\label{sec:numerics}

    \cu{This section presents several simulations for proposed methods. Particularly, we implement Algorithm~\ref{alg:Delta} for the logistic regression problem on both synthetic and real data sets. Also, we show the performance of Algorithm~\ref{alg:Delta} on a family of functions recently described as difficult for all tensor methods~\cite{nesterov2021implementable}. We focus on the case where~$p=3$ for which we have efficient methods for solving the auxiliary subproblem~\cite[Section 5]{nesterov2021implementable}. Finally, we present the performance results for the entropy regularized optimal transport and minimal mutual information problems.}
    
    \subsection{Logistic Regression}
    
        For the logistic regression problem, we are given a set of~$d$ data pairs~$\{y_i,w_i\}$ for~$1\leq i\leq d$, where~$y_i \in \{1, -1\}$ is the class label of object~$i$, and~$w_i \in \mathbb{R}^n$ is the set of features of object~$i$. After the dimension and number of data points are set. The optimal point is generated as~$x^*$ composed in each dimension as samples from a uniform distribution in the range~$[-1,1]$. Each dimension per data sample is also generated as samples from a uniform distribution in the range~$[-1,1]$ with the last feature set to~$1$ for all data points. The label is generated as the sign of the products of the features and~$x^*$. Finally, labels are flipped with a probability of~$0.01$. We are interested in finding a vector~$x$ that solves the following optimization problem 
        \begin{align}\label{eq:logistic loss}
                \min_{x \in \mathbb{R}^n} \frac{1}{d}\sum\limits_{i = 1}^d \ln\Bigl(1 + \exp\bigl(-y_i\langle w_i, x \rangle\bigr)\Bigr). 
        \end{align}
        
        Figure~\ref{fig:sinthetic} shows the gradient norm of the logistic regression function at the points generated by Algorithm~\ref{alg:Delta}. Initially, we show the results for synthetic data where~$d=100$ and~$n=10$. We focus on showing the results for different values of~$\e$. Here by Iterations we mean the number of iterations of Algorithm~\ref{alg:MSN} inside Algorithm~\ref{alg:Delta}, line 3. For implementation simplicity, in addition to stopping criterion~$A_{N_k} \ge \frac{2}{\mu}$, if the gradient is no longer decreasing, we apply the restarting of Algorithm~\ref{alg:MSN} after~$500$ iterations.   
        
        \begin{figure}[th!]
        	\centering
        	\includegraphics[width=0.7\linewidth]{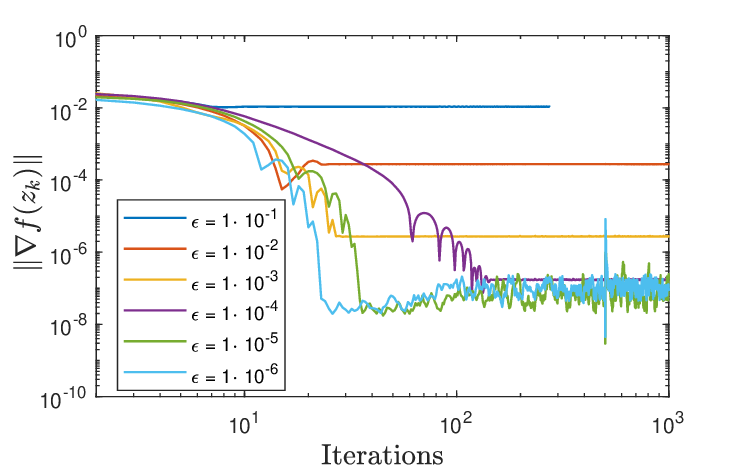}
        	\caption{Gradient norm at the iterations generated by Algorithm~\ref{alg:Delta} on synthetic data for various values of~$\e$. }
        	\label{fig:sinthetic}
        \end{figure}
        
        Figure~\ref{fig:real} shows the gradient norm of the logistic regression function at the points generated by Algorithm~\ref{alg:Delta}. In this case, we use the \texttt{Mushroom}, \texttt{A9A}, \texttt{Covertype} and \texttt{IJCNN1} datasets from ~\cite{Dua:2019} with a fixed value of~$\e = 1\cdot 10^{-5}$.
        
        \begin{figure}[th!]
        	\centering
        	\includegraphics[width=0.7\linewidth]{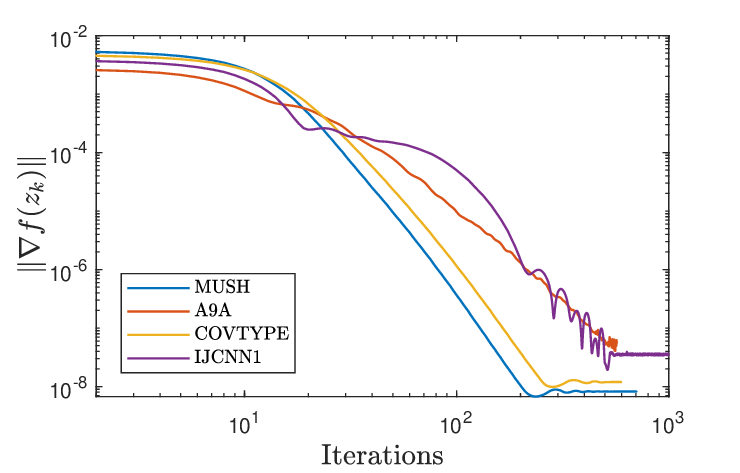}
        	\caption{Gradient norm at the iterations generated by Algorithm~\ref{alg:Delta} on real data sets from~\cite{Dua:2019} with~$\e=1\cdot 10^{-5}$. }
        	\label{fig:real}
        \end{figure}
        
        \subsection{A family of difficult functions}
        
        Next, we analyze the performance of the proposed algorithm on a universal parametric family of objective functions, which are difficult for all tensor methods~\cite{nesterov2021implementable,grapiglia2020tensor} defined as
        \begin{align}\label{eq:bad_functions}
                f_m(x) = \eta_{p + 1} \left ( A_m x \right ) - x_1,
        \end{align}
        where, for integer parameter~$p \, \ge \, 1$,
       ~$\eta_{p + 1} (x) = \frac{1}{p+1} \sum\limits_{i = 1}^ n
        | x_i |^{p + 1}$,~$2 \, \le \, m \, \le \, n$,~$x \, \in \, \mathbb{R}^n$,
       ~$A_m$ is the~$n \times n$ block diagonal matrix:
        \begin{align}
              A_m = 
        \left ( \begin{array}{cc}
        U_m & 0 \\
        0 & I_{n - m}
        \end{array}
        \right ), \quad \text{with} \quad U_m = 
        \left ( \begin{array}{ccccc}
        1 & -1 & 0 & \hdots & 0 \\
        0 & 1 & -1 &  \hdots & 0 \\
        \vdots & \vdots & \ddots &  & \vdots \\
        0 & 0 &  \hdots & 1 & -1 \\
        0 & 0 &  \hdots & 0 & 1
        \end{array}
        \right ),  
        \end{align}
        and~$I_n$ is the identity~$n \times n$-matrix. For a detailed description of the high-order derivatives of this class of functions and its optimality properties, see~\cite{nesterov2021implementable}.
        
        Finally, Figure~\ref{fig:bad_fun} shows the performance results of Algorithm~\ref{alg:Delta} on the family of functions in \eqref{eq:bad_functions} with~$p=3$ and various values of parameters~$m=n$ with~$\e = 1\cdot 10^{-5}$.
        
        \begin{figure}[th!]
        	\centering
        	\includegraphics[width=0.7\linewidth]{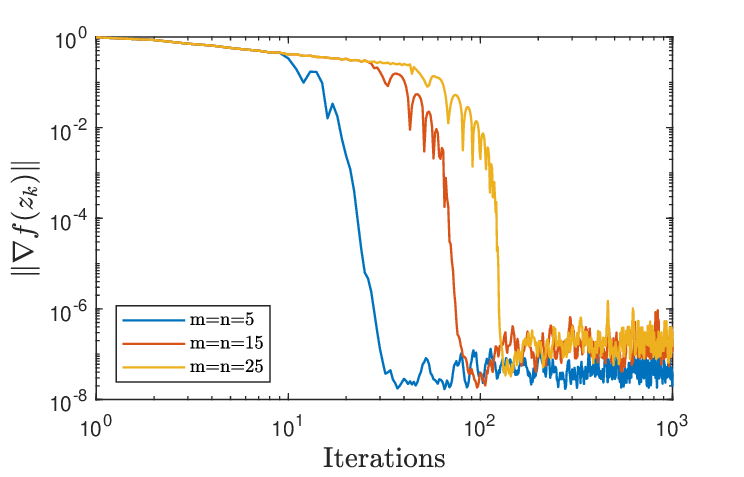}
        	\caption{Gradient norm at the iterations generated by Algorithm~1 on the family of functions in \eqref{eq:bad_functions} with~$p=3$ and various values of parameters~$m=n$ with~$\e = 1\cdot 10^{-5}$. }
        	\label{fig:bad_fun}
        \end{figure}

    \subsection{The entropy regularized optimal transport problem}
    
        We now go back to the entropy-regularized optimal transport problem in~\eqref{eq:wass} and present some numerical experiments of the proposed method applied to its dual problem in~\eqref{eq:OT-PD}.
        
        \begin{align*}
            \phi(\lambda) & =  \textbf{smax}_\gamma(A^\intercal\lambda-M) + \la \lambda, b\ra,
        \end{align*}
        
        Initially, we recall some properties of the Softmax function, which will be useful for the implementation; for a complete analysis of such properties, see~\cite{bullins2019higher}.
        
        \begin{proposition}
            Consider the softmax function in~\eqref{eq:softmax}, and define the function~$f(x)= \textbf{smax}_\gamma(Ax- b) -  \la \lambda, b\ra$. Then the following properties hold:
            \begin{align*}
            \nabla \textbf{smax}_\mu(x)_i &  = \exp\left( \frac{x_i}{\mu}\right)\Big/\left(\sum_i \exp\left( \frac{x_i}{\mu}\right)\right)\\
            \nabla^2 \textbf{smax}_\mu(x) & = \frac{1}{\mu} \big( \text{diag} \big ( \nabla\textbf{smax}_\mu(x) \big)- \nabla\textbf{smax}_\mu(x)\nabla \textbf{smax}_\mu(x)^\intercal \big) \\
            \nabla^3 \textbf{smax}_\mu(x)[h,h] & = \frac{1}{\mu} \big( \nabla^2\textbf{smax}_\mu(x)[h^2] - 2\langle \nabla\textbf{smax}_\mu(x),h \rangle\nabla^2\textbf{smax}_\mu(x)[h] \big),
        \end{align*}
        and
        \begin{align*}
            \nabla f(x) & = A^\intercal\nabla \textbf{smax}_\mu(Ax-b)\\
            \nabla^2 f(x) & = A^\intercal\nabla^2 \textbf{smax}_\mu(Ax-b)A\\
            \nabla^3 f(x)[h,h] & = A^\intercal\nabla^3 \textbf{smax}_\mu(Ax-b) [Ah,Ah].
            \end{align*}
        \end{proposition}

        \subsubsection{Discrete probability distributions}
        
            Next, we present the numerical results for the computation of the optimal (entropy-regularized) transport plan between two discrete probability distributions using the near-optimal third-order method in Algorithm~\ref{alg:Delta}. We construct two discrete distributions as the mixture of three randomly generated Gaussian distributions, each on bounded support~$[-5,5]$ with~$n=100$. We select the regularization parameter to~$\gamma=0.1$, which is common for these applications. Figure~\ref{fig:OT_results} shows three examples of the resulting transport plan obtained by Algorithm~\ref{alg:Delta} for three different pairs of distributions, and the corresponding distributions are shown as the marginals of the transport plan. Figure~\ref{fig:OT_norms} shows the corresponding norms of the gradients, evaluated at each iteration of Algorithm~\ref{alg:Delta} for the three problems shown in Figure~\ref{fig:OT_results}.

            \begin{figure}[ht!]
            	\centering
            	\subfloat[]{%
            		\includegraphics[width=0.3\textwidth]{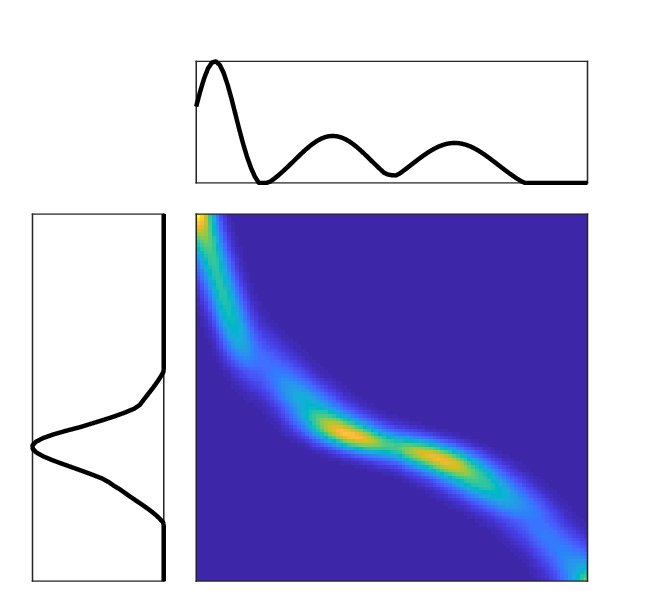}}
            	\subfloat[]{%
            		\includegraphics[width=0.3\textwidth]{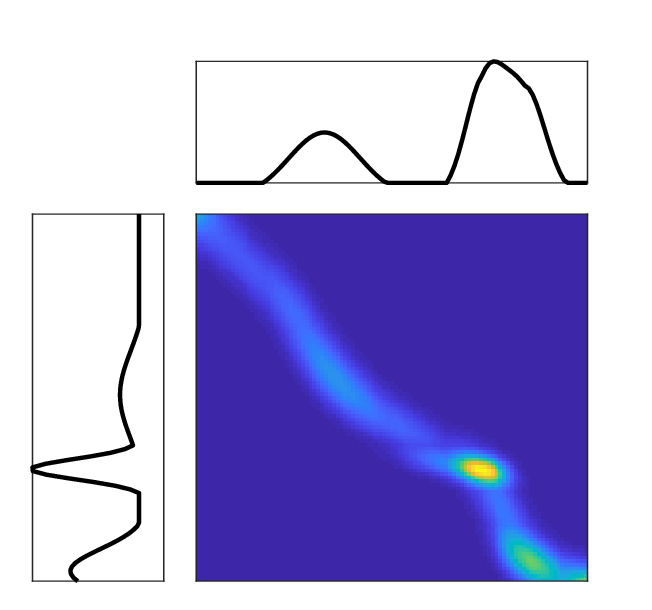}}
        		\subfloat[]{%
            		\includegraphics[width=0.3\textwidth]{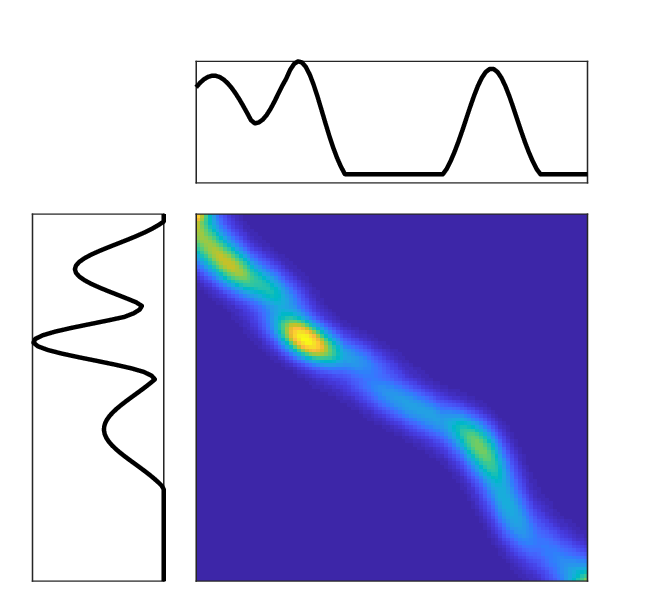}}
            	\caption{Three separate examples of the resulting transport plan obtained by Algorithm~\ref{alg:Delta}. The two distributions are shown on the left and top of the transport plan.} \label{fig:OT_results}
            \end{figure}
            
            \begin{figure}[ht!]
            	\centering
            	\subfloat[]{%
            		\includegraphics[width=0.3\textwidth]{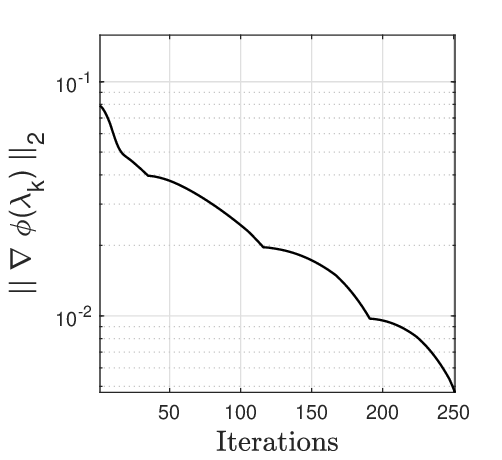}}
            	\subfloat[]{%
            		\includegraphics[width=0.3\textwidth]{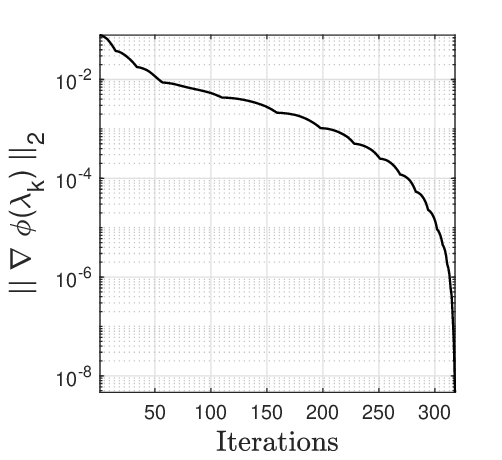}}
        		\subfloat[]{%
            		\includegraphics[width=0.3\textwidth]{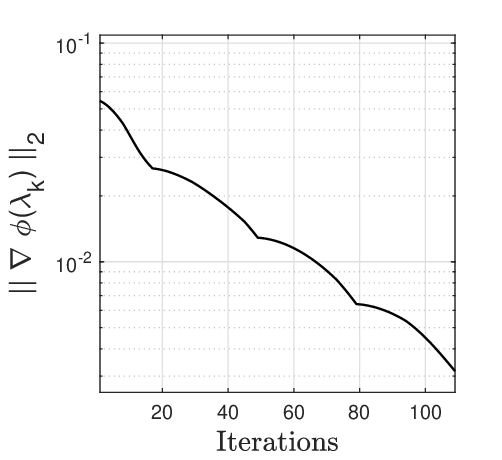}}
            	\caption{Norm of the gradient at each iteration of Algorithm~\ref{alg:Delta} for the three problems shown in Figure~\ref{fig:OT_results}.} \label{fig:OT_norms}
            \end{figure}

        \subsubsection{MNIST}\label{sec:mnist}
            Finally, we provide the results of experiments for transportation plan computation between two MNIST images of handwritten digits. Here we compare the results of Algorithm~\ref{alg:R} \po{(GN)} and Algorithm~\ref{Alg:PDATM} \po{(PDATM)}. 
            These methods represent two approaches to tackle this problem: gradient norm minimization of dual function and primal-dual method. 
            As we mentioned earlier, comparing these two approaches is our main motivation in this paper. 
            \po{Additionally, we provide the results of the Algorithm \ref{alg:MSN} (ATD) applied to dual function.
            In detail, we use it to minimize dual function until we achieve the prescribed $\e$-approximate solution for constraint and duality gap.
            In other words, we use it as a primal-dual method.}

            In this experiment, we take the images presented in Figure~\ref{fig:MNIST images} as initial histograms. The size of each picture is~$28 \times 28$ pixels. We reshape these images to vectors of size~$n = 28^2 = 784$.

            \po{To perform the inner ``tensor'' step \eqref{eq:tensor step} inside each of the considered algorithms, we use the method developed in~\cite{nesterov2021implementable}. We use its inexact modification from \cite{nesterov2021superfast} and look for the points from the following neighborhood:
            \begin{equation}\label{eq:inner problem neighborhood}
                \mathcal N_{p, M}^\alpha(x) \equiv \left\{ T \in \R^n: \|\nabla \Omega_{x, p, M}(T)\|_2 \le \alpha \| \nabla f(T) \|_2 \right\},
            \end{equation}
            where we choose the same size of the neighborhood as in \cite{nesterov2021superfast}: $\alpha = \frac{1}{2p} = \frac{1}{6}$.
            Thus, we run our inner-problem method until we reach the point $T \in \R^n$, inside the set defined in \eqref{eq:inner problem neighborhood}.
            }
            
            We select the regularization parameter~$\gamma = 0.5$ and Lipschitz constant~\po{$M_p = 0.5$}. We want to emphasize that we do not use theoretical estimation of~$M_p$ from Proposition~\ref{them:smax} because it gives too pessimistic upper bound, which results in too slow convergence. 
            \po{
            We start from zero and stop the optimization process for Algorithms \ref{Alg:PDATM} and \ref{alg:MSN} when both constraint and duality gap become smaller or equal than $\e = 0.001$.
            For Algorithm \ref{alg:R}, we stop when the norm of the gradient of dual function is less or equal than $\min\{\e; \frac{\e}{2R}\}$.
            To choose $R$, we use the following lemma.
            \begin{lemma}[Lemma 11 in \cite{guminov2021combination}]
                Let $M \in \R_+^{n \times n}$ be a transportation matrix, $p, q \in \Sigma_n$ be two histograms. Then, there exists a solution $(\xi^*, \eta^*)$ of \eqref{eq:OT-PD} such that
                \[
                    \|(\xi^*, \eta^*)\|_2 \le 
                    R := \sqrt{N / 2} \left( \|M\|_\infty - \frac{\gamma}{2} \ln \min_{i, j} \{p_i, q_j\} \right).
                \]
            \end{lemma}
            }
            \begin{figure}[t]
                \centering
                \subfloat[]{
                     \includegraphics[width=0.5\textwidth]{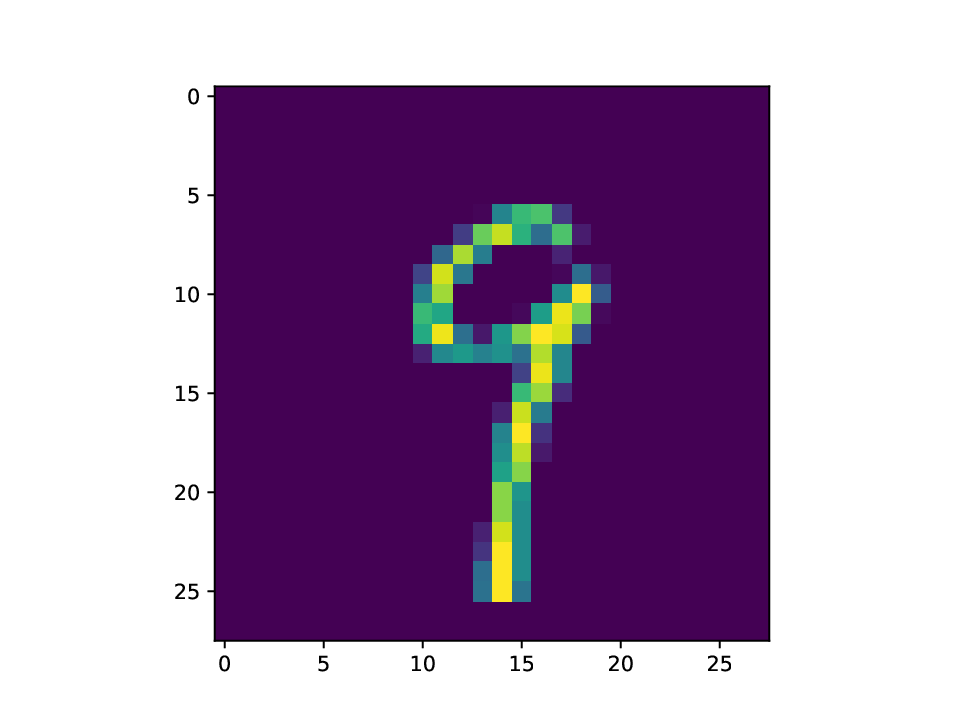}
                }
                \subfloat[]{
                     \includegraphics[width=0.5\textwidth]{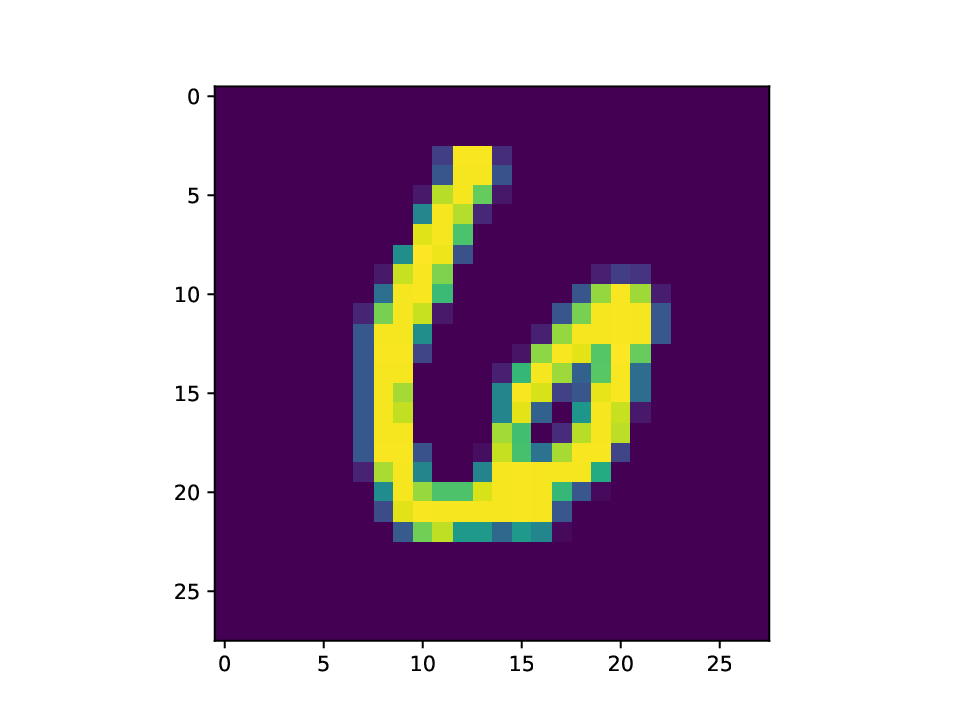}
                }
                \caption{Initial and target images for optimal transport problem}
                \label{fig:MNIST images}
            \end{figure}

            \begin{figure}[t]
                \centering
                \subfloat[]{
                    \includegraphics[width=0.5\linewidth]{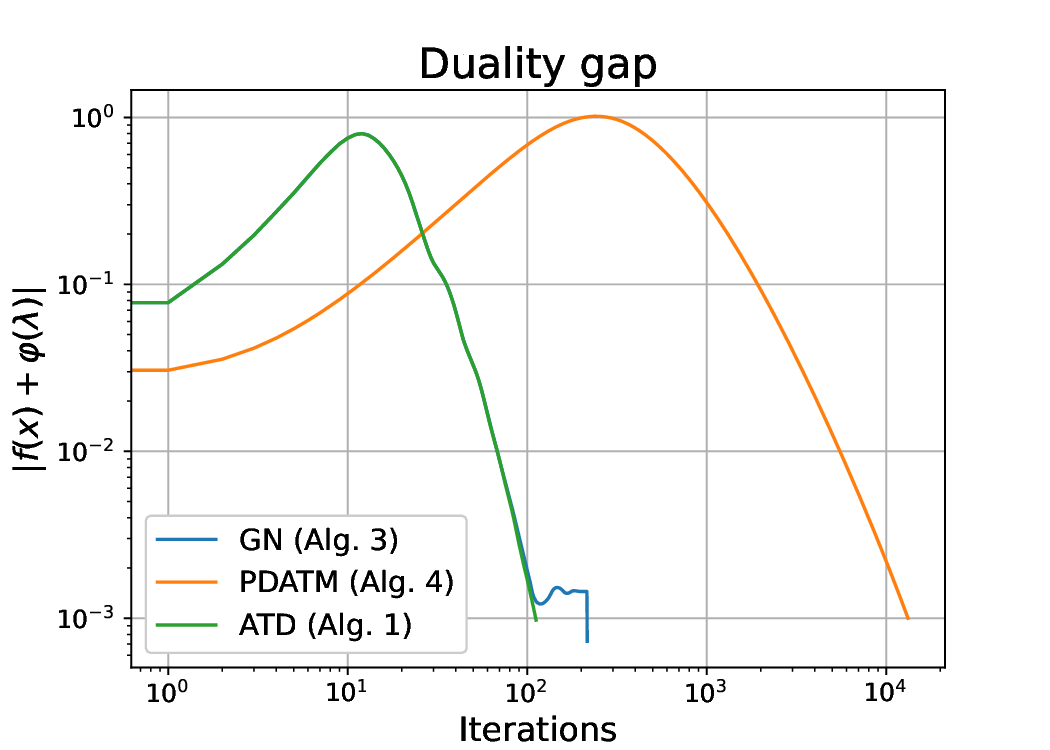}
                    \label{fig:ot dual gap}
                }
                \subfloat[]{
                     \includegraphics[width=0.5\textwidth]{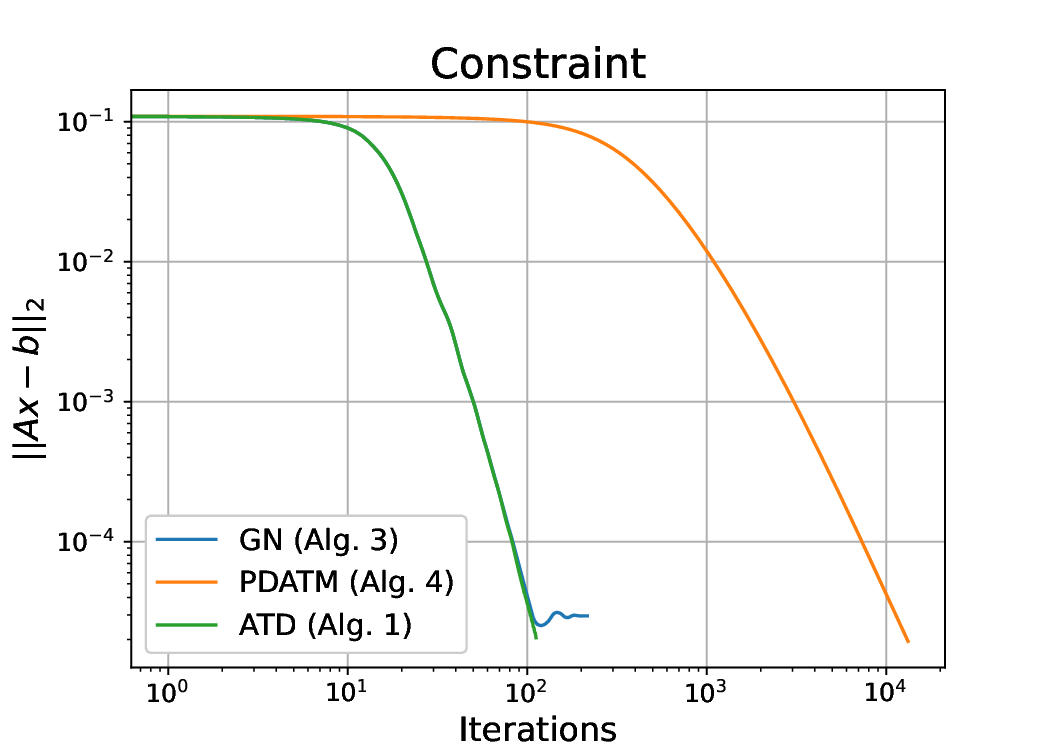}
                     \label{fig:ot constraint}
                }
                \caption{Duality gap and equality constraint convergence for Algorithms~\ref{alg:R}, ~\ref{Alg:PDATM} and \ref{alg:MSN}.}
                \label{fig:ot dual gap and constraint}
            \end{figure}

            \begin{figure}[ht!]
                \centering
                \subfloat[]{
                    \includegraphics[width=0.5\textwidth]{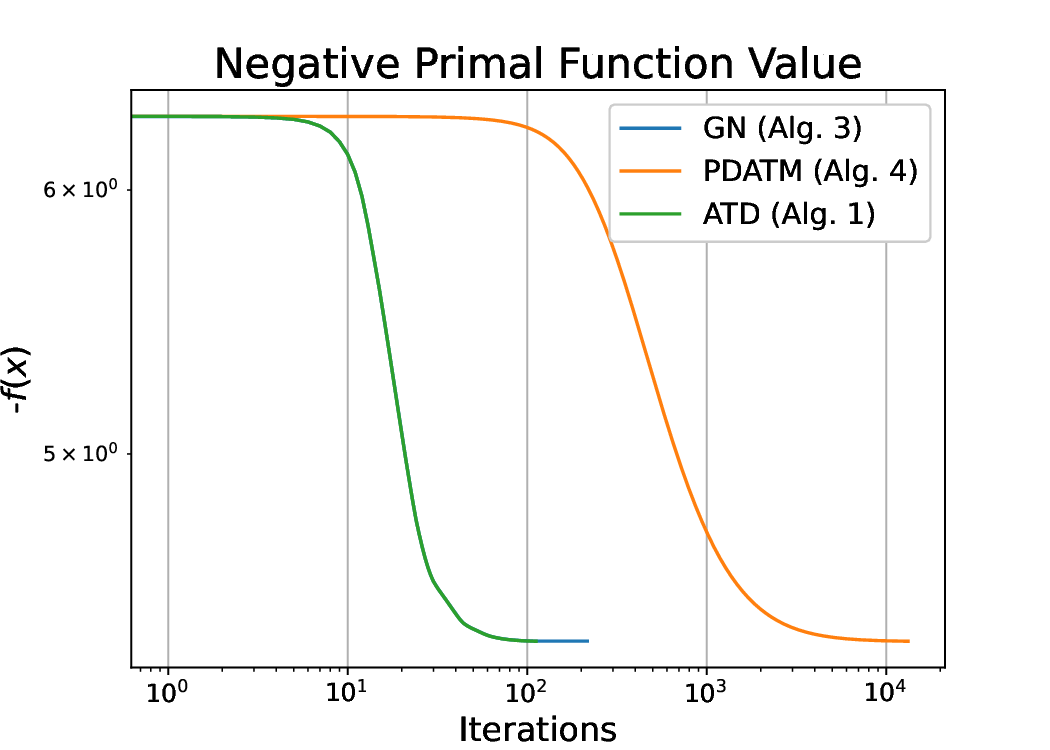}
                    \label{fig:f}
                }
                \subfloat[]{
                    \includegraphics[width=0.5\textwidth]{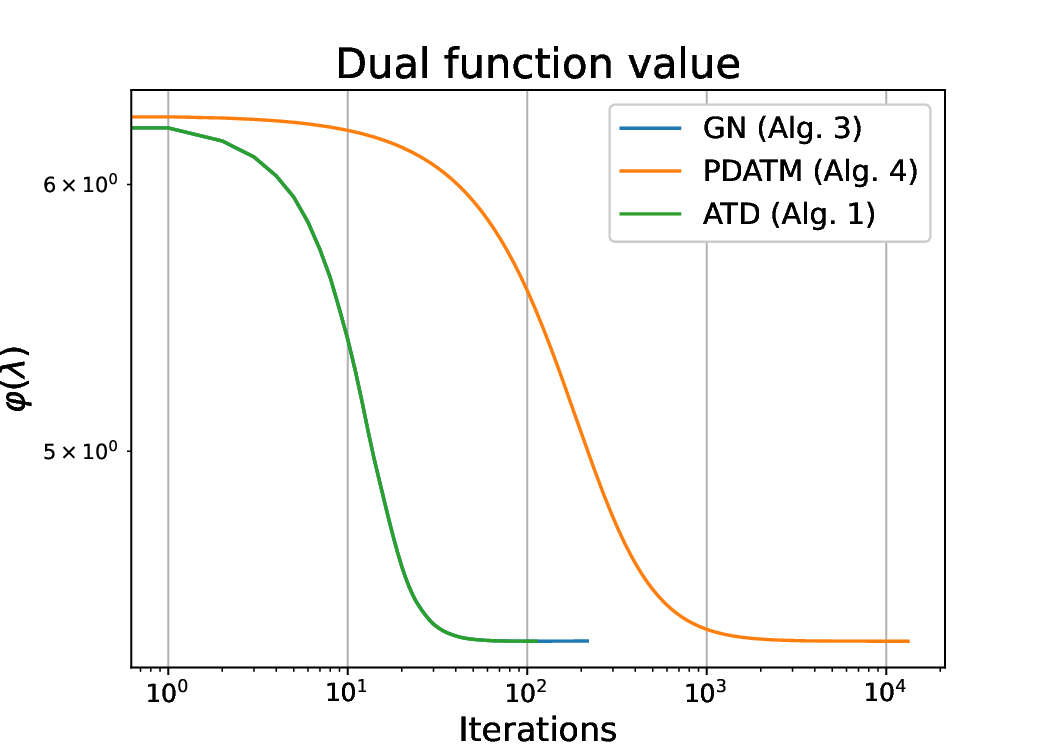}
                    \label{fig:phi}
                }
                \caption{Negative primal and dual function values for Algorithms~\ref{alg:R} and~\ref{Alg:PDATM}}
                \label{fig:f and phi}
            \end{figure}

            In Figure~\ref{fig:ot dual gap}, we show the duality gap convergence, and in Figure~\ref{fig:ot constraint}, we show the convergence results for equality constraints. Additionally, we provide the results for values of~$f(x)$ and~$\vp(\lambda)$ on Figure~\ref{fig:f and phi}. In Figure~\ref{fig:f}, we show a negative value of~$f(x)$ since otherwise, we could not plot it in a log scale.

            \po{
            At first, we compare the results of Algorithm \ref{alg:R} and Algorithm \ref{Alg:PDATM}, and then we discuss the performance of Algorithm \ref{alg:MSN}.
            All these figures show that Algorithm~\ref{alg:R} outperforms Algorithm~\ref{Alg:PDATM} in convergence speed. 
            The little spikes on Figures~\ref{fig:ot dual gap} and \ref{fig:ot constraint} at the end of the blue line are the result of restarts of Algorithm~\ref{alg:MSN} inside Algorithm~\ref{alg:R}. 
            The sharp drop at the end of the blue line is due to the final tensor step in Line 6 of Algorithm~\ref{alg:R}. 
            }


            We mentioned earlier, \po{and it is noticeable} from pseudocode, that Algorithm~\ref{Alg:PDATM} is a \po{\textit{direct}} method, which means that it does not use restarts, regularization, or binary search. These modifications introduce difficulties in implementing the method and may slow it down. For example, restarts usually give a too-pessimistic upper bound for the number of iterations of the inner method. Nevertheless, the Primal-Dual method needs only Lipschitz constant and estimation accuracy. Despite all these facts, it still loses to Algorithm~\ref{alg:R} both in practice and theory (see Remark~\ref{rem:pd and gd comparison}). 
            \po{
            Both Algorithms \ref{alg:R} and Algorithm \ref{Alg:PDATM} have a single hyperparameter -- Lipschitz constant estimation $M_p$. 
            But since it is more a characteristic of the problem than a particular algorithm hyperparameter, both algorithms should be the same. 
            We conducted several experiments where we compare Algorithm \ref{alg:R} and Algorithm \ref{Alg:PDATM} with the same parameter $M_p$ from the set $\{0.01, 0.1, 0.5, 1, 5\}$. 
            The overall picture is the same, and Algorithm \ref{alg:R} outperforms Algorithm \ref{Alg:PDATM}.
            Figure \ref{fig:ot dual gap and constraint} shows that Algorithm \ref{alg:R} restarts only in the end, which does not affect the overall result. 
            Since the theoretical estimate for the number of iterations, after which we should restart our algorithm, is usually too pessimistic, we tried to perform restarts manually after every 25 or 50 iterations of the inner algorithm.
            However, it only worsened the convergence of the whole Algorithm \ref{alg:R}.
            }

            \po{
            Finally, algorithm \ref{alg:MSN} behaves the same way as Algorithm \ref{alg:R}, and in the end, it converges even faster.
            The reason it covers in the restart condition $A_{N_k} \ge \frac{4}{\mu}$: Algorithm \ref{alg:R} chooses to restart when it has almost achieved the solution.
            Again, one can see it as little spikes at the end of the blue line.
            Since, in our case, $\mu = \frac{\e}{4R}$, we can not just directly change the value of $\mu$.
            We can do this only through $\e$ because $R$ is theoretically estimated and specific to chosen objective problem.
            However, both the increase and decrease of $\e$ in our experiments showed similar results.
            Thus, to address this issue, we conduct additional experiments on a strongly convex objective in the next subsection, where we can directly specify the constant of strong convexity $\mu$.
            }


    \po{
    \subsection{Minimal Mutual Information}\label{sec:mmi}
    }

        \po{
        The experiments for entropy regularized optimal transport were not representative of the performance of Algorithm \ref{alg:R} compared to Algorithm \ref{alg:MSN}; in this subsection, we look at the Minimal Mutual Information problem, which is strongly convex.
        This problem is defined as follows:
        \begin{equation}\label{eq:mmi}
             \min_{x \in \Sigma_n} \left\{ f(x):= \frac{L}{2}\|A x-b\|_2^2+ \mu \sum_{k=1}^n x_k \ln x_k \right\},
        \end{equation}
        where $\Sigma_n$ is $n$-dimensional standard simplex.}

        \po{
        If we consider this problem in different space $\{(x, z) | z = Ax, x \in \Sigma_n\}$, then it becomes optimization problem with linear equality constraints
         \begin{equation}\label{eq:mmi in (x, z) space}
             \min _{\substack{x \in \Sigma_n \\ z=A x}} \left\{ f(x, z) := \frac{L}{2}\|z - b\|_2^2 + \mu \sum_{k=1}^n x_k \ln x_k \right\}.
        \end{equation}
        The dual problem to \eqref{eq:mmi in (x, z) space} looks as follows
        \begin{equation}\label{eq:mmi dual}
            \min _{\lambda \in \mathbb{R}^m} \left\{ \vp(\lambda) := \mu \ln \left(\sum_{i=1}^n \exp \left(\frac{\left[-A^T \lambda\right]_i}{\mu}\right)\right)+\frac{1}{2 L}\left(\|\lambda+b\|_2^2-\|b\|_2^2\right) \right\}.
        \end{equation}
        }

        \po{
        Since, in this case, dual objective \eqref{eq:mmi dual} is initially strongly convex with constant $\frac{1}{L}$, we do not need additional regularization in Algorithm \ref{alg:R}, and we use $f_\mu \equiv \vp(\lambda),\ \mu = \frac{1}{L}$.
        We do not provide any additional analysis for the case when the objective of Algorithm \ref{alg:R} is initially strongly convex because the differences with proofs of Theorem \ref{thm:arg residual} are minor.
        The resulting convergence rates can be seen in Remark \ref{rem:extension to strongly convex case}.
        }

        \po{
        In our experiments we used dataset "housing" from \cite{CC01a}, scaled to $[-1, 1]$, which we then transfered to $[0, 1]$.
        We tested several values of $M_p$. 
        The overall picture was the same, but the convergence of all the algorithms was faster for smaller values of $M_p$, and the resulting picture was not so evident.
        Thus, we decided to choose $M_p = 100$.
        Additionally, we have tested several values of $L$ and $\mu$.
        Again, the overall picture was the same, but it took too long for some values to converge for all the methods.
        Thus, we chose $\mu = 1$, $L = 10$.
        We start from zeros and stop the optimization process when the approximation error achieves $\e = \e_f = \e_{eq}  = 0.01$.
        We derive an estimate of $R$ from strong convexity:
        \[
            \|\lambda^*\|_2 \le \frac{\|\nabla \vp(0) - \nabla \vp(\lambda^*)\|_2}{\mu} = \frac{\|\nabla \vp(0)\|_2}{\mu} = R.
        \]
        }

        \begin{figure}[ht!]
            \centering
            \subfloat[]{
                 \includegraphics[width=0.5\textwidth]{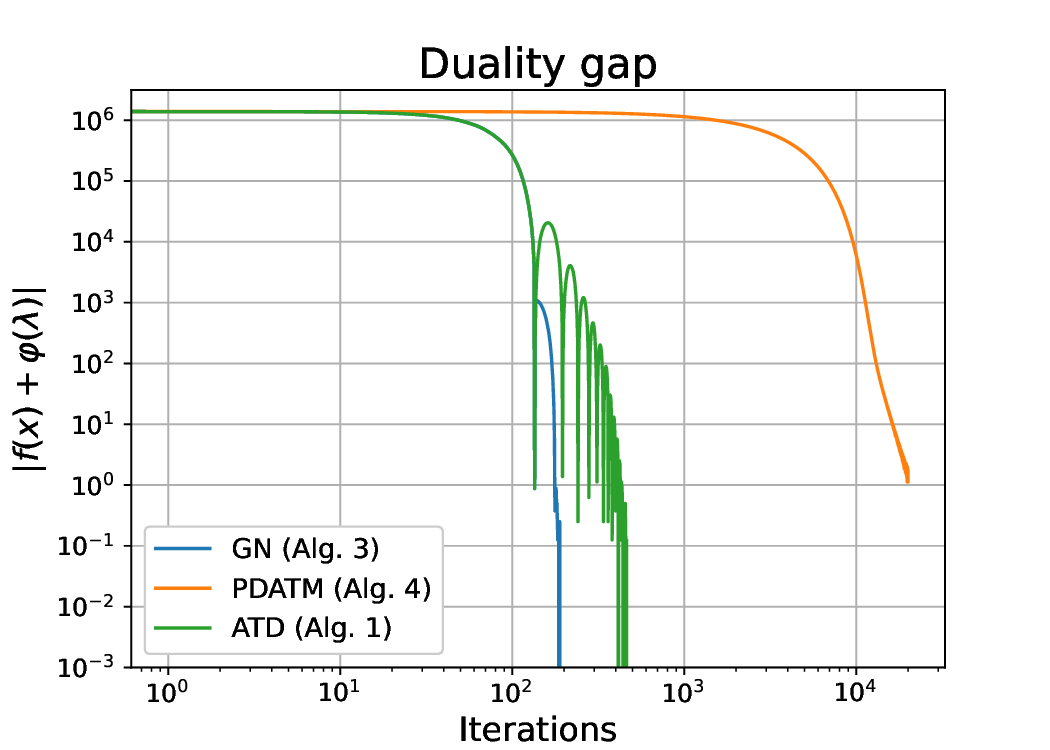}
                 \label{fig:mmi dual gap}
            }
            \subfloat[]{
                 \includegraphics[width=0.5\textwidth]{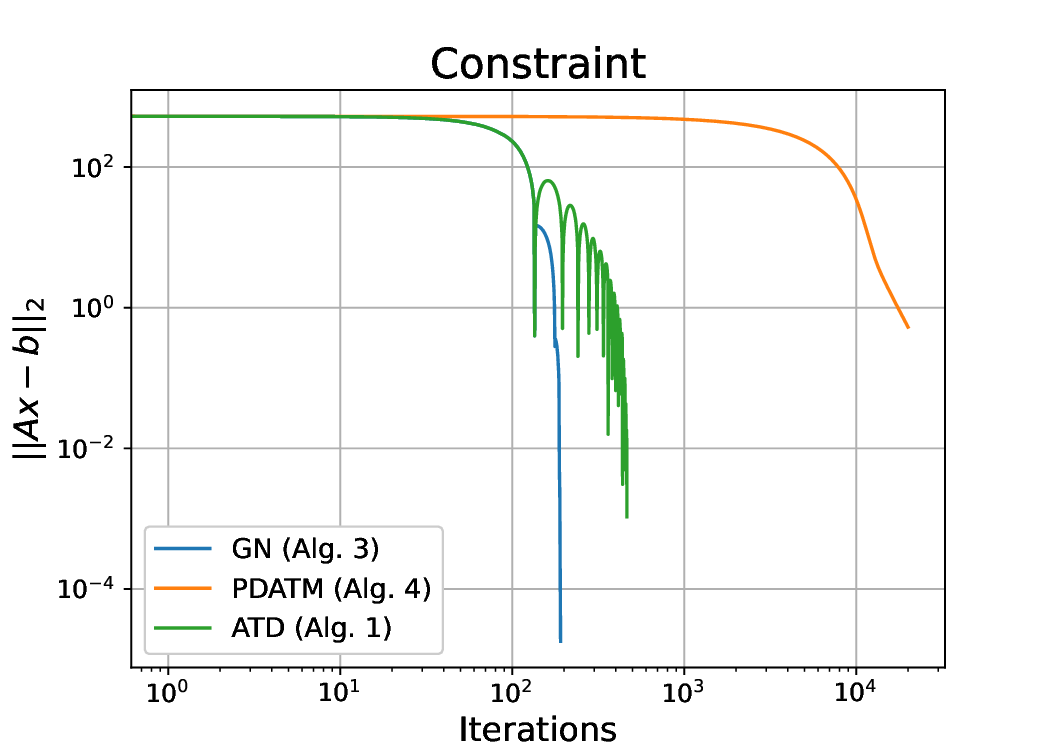}
                 \label{fig:mmi constraint}
            }
            \caption{\po{Duality gap and constraint convergence for Algorithms \ref{alg:R}, \ref{Alg:PDATM} and \ref{alg:MSN}}}
            \label{fig:mmi dual gap and constraint}
        \end{figure}

        \begin{figure}[ht!]
            \centering
            \subfloat[]{
                 \includegraphics[width=0.5\textwidth]{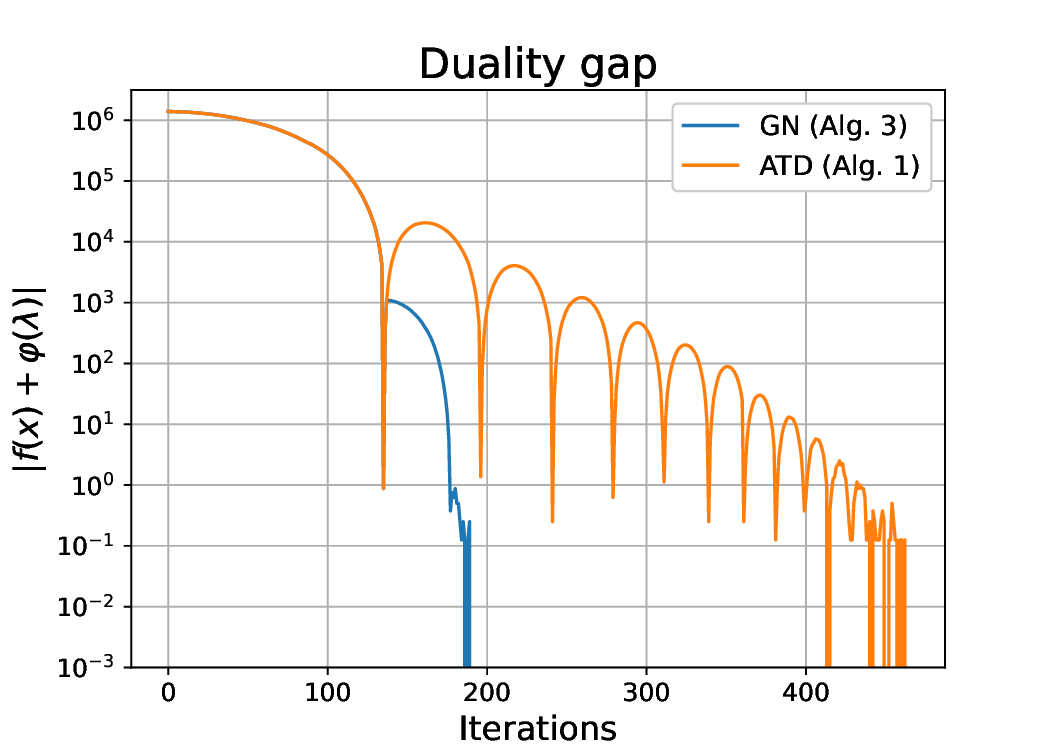}
                 \label{fig:mmi dual gap (gn vs tm)}
            }
            \subfloat[]{
                 \includegraphics[width=0.5\textwidth]{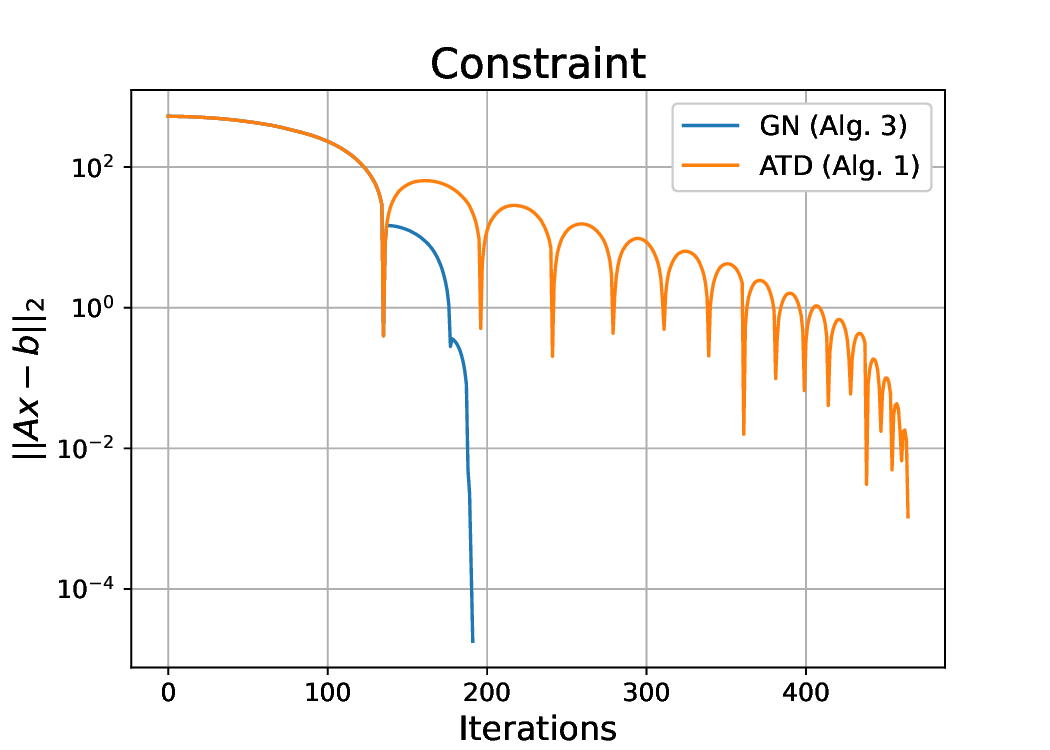}
                 \label{fig:mmi constraint (gn vs tm)}
            }
            \caption{\po{Duality gap and constraint convergence for Algorithms \ref{alg:R} and \ref{alg:MSN}}}
            \label{fig:mmi dual gap and constraint (gn vs tm)}
        \end{figure}
        
        \po{
        We show the comparison of Algorithms \ref{alg:R}, \ref{Alg:PDATM} and \ref{alg:MSN} in Figure \ref{fig:mmi dual gap and constraint}, and then take a closer look at Algorithms \ref{alg:R} and \ref{alg:MSN} in Figure \ref{fig:mmi dual gap and constraint (gn vs tm)}.
        Here Algorithm \ref{alg:MSN} minimizes dual function \eqref{eq:mmi dual} and recalculates the value of the primal function and equality constraint on every iteration until it reaches approximation error of $\e$ both for duality gap and constraint.
        }

        \po{
        Again, like in MNIST experiments, we can see that Algorithm \ref{alg:R} for gradient norm minimization (GN) outperforms Algorithm \ref{Alg:PDATM}, that is, primal-dual accelerated tensor method (PDATM).
        We limit the maximal number of steps for Algorithm \ref{Alg:PDATM} to 20000, so it stops when it reaches this limit and does not converge to the area of approximation error neither for dual gap (Figure \ref{fig:mmi dual gap}) nor for constraint (Figure \ref{fig:mmi constraint}).
        Thus, Algorithm \ref{alg:R} converges more than 100 times faster.
        }

        \po{
        Figure~\ref{fig:mmi dual gap and constraint (gn vs tm)} shows that Algorithm \ref{alg:MSN} (ATD) performs the same way as Algorithm \ref{alg:R} until the first restart.
        Then they both start jumping: Algorithm \ref{alg:R} due to restarts, and Algorithm \ref{alg:MSN} -- due to the nature of used acceleration.
        But, our proposed Algorithm \ref{alg:R} converges almost 2.5 times faster.
        This result shows that our proposed gradient norm minimization framework makes sense when it restarts early.
        }
            
        \po{
        \begin{remark}
            The above results show that Monteiro-Svaiter acceleration~\cite{monteiro2013accelerated}, which we used in Algorithm \ref{alg:R}, covers all the implementational drawbacks of Algorithm \ref{alg:R}, which results in better numerical convergence compared to Algorithm \ref{Alg:PDATM}.
            As we mentioned at the end of Remark~\ref{rem:extensions}, in Algorithm~\ref{Alg:PDATM}, we use Nesterov acceleration, which gives a worse convergence rate than Monteiro-Svaiter acceleration.
            Future work should consider a primal-dual method with Monteiro-Svaiter acceleration and compare its performance numerically with Algorithms \ref{alg:R} and \ref{Alg:PDATM}. This will make the Primal-Dual method not that straightforward because at least it will introduce additional linear search, which comes with Monteiro-Svaiter acceleration.
        \end{remark}
        }

        \po{
        \begin{remark}
            Comparison of results of Algorithms \ref{alg:R} and \ref{alg:MSN} on convex entropy-regularized optimal transport (Section \ref{sec:mnist}) and on strongly convex MMI problem (Section \ref{sec:mmi}) showed that it is beneficial to use Algorithm \ref{alg:R}, when $\mu$ is not too small.
            Otherwise, the first restart would be done when the method almost achieved the solution or would not be done at all.
            For example, in Section \ref{sec:mnist} we had $\e = 0.001$, $R \simeq 100\ \Rightarrow \mu = \e / (4R) \sim 10^{-5}$, and in Section \ref{sec:mmi} $\mu = 1 / L = 10^{-1}$.
            That is why the second case method restarts much earlier than it achieves its solution. 
        \end{remark}
        }
\section{Conclusions}\label{sec:conclusions}
    \po{
    This paper considers minimization problems with linear equality constraints.
    There are two ways to solve this type of problem: find a stationary point of dual function and reconstruct the primal solution or use the primal-dual method, which optimizes both primal and dual function simultaneously.
    We consider both approaches.
    Firstly, we propose two high-order methods for gradient norm minimization.
    These methods have optimal convergence rates up to multiplicative logarithmic factors.
    Secondly, we propose a high-order primal-dual accelerated tensor method that uses Nesterov's acceleration.
    Finally, we compare these two approaches with each other both in theory and in practice.
    Additionally, we numerically compare the proposed methods with the primal-dual version of the near-optimal tensor method for convex optimization.
    }

    
\section*{Funding}

    This work was supported by Ministry of Science and Higher Education grant No. 075-10-2021-068.
    The work by P. Dvurechensky was funded by the Deutsche Forschungsgemeinschaft (DFG, German Research Foundation) under Germany's Excellence Strategy - The Berlin Mathematics
Research Center MATH$^+$ (EXC-2046/1, project ID: 390685689).
    The work of C.A. Uribe is supported in part by the National Science Foundation under Grant No. 2211815 and No. 2213568.
\bibliographystyle{tfs}
\bibliography{references}

\end{document}